\documentclass[11pt,letterpaper]{amsart}
\usepackage{enumerate}
\usepackage{amssymb,amsmath,graphicx,amsthm, fullpage}
\usepackage{color}
\usepackage{soul}

\def\T{\text}
\newcommand{\Om}{\Omega}

\newcommand{\no}[1]{\|{#1}\|}

\newcommand{\eps}{\varepsilon}

\newcommand{\n}{\noindent}

\newcommand{\R}{\mathbb R}

\newcommand{\N}{\mathbb N}
\newcommand{\C}{\mathbb C}
\DeclareMathOperator{\supp}{supp}

\DeclareMathOperator{\Rre}{Re}

\newcommand{\bd}{\textrm{b}}

\newcommand{\p}{\partial}
\newcommand{\les}{\lesssim}
\newcommand{\ges}{\gtrsim}

\newcommand{\dbar}{\bar\partial}

\newcommand{\nn}{\nonumber}
\newcommand{\ep}{\epsilon}

\def\di{\partial}
\def\dib{\bar\partial}

\def\6#1{{\mathcal{#1}}}

\def\B{\6B}

\def\M{\6M}

\def\Im{\T{Im}\,}
\def\Re{\T{Re}\,}

\numberwithin{equation}{section}

\def\T{\text}

\theoremstyle{plain}

\newtheorem{theorem}{Theorem}[section]

\newtheorem{lemma}[theorem]{Lemma}

\newtheorem{proposition}[theorem]{Proposition}

\theoremstyle{definition}

\newtheorem{definition}[theorem]{Definition}

\theoremstyle{remark}

\newtheorem{remark}[theorem]{Remark}

\begin{document}
	
	\title[Regularity properties for the Bergman  projection]{Local regularity of the Bergman projection on a class of pseudoconvex domains of finite type}
	
	\author{Tran Vu Khanh and Andrew Raich}   
	\address{T.~V.~Khanh}
	\address{Department of Mathematics, International University, Vietnam National University-Ho Chi Minh City, Ho Chi Minh City, Vietnam}
	\email{tvkhanh@hcmiu.edu.vn}
	\address{A. Raich}
	\address{Department of Mathematical Sciences, SCEN 327, 1 University
		of Arkansas, Fayetteville, AR 72701, USA} 
	\email{araich@uark.edu}
	\begin{abstract} The purpose of this paper is to prove that if a pseudoconvex domains $\Omega\subset\mathbb{C}^n$ satisfies Bell-Ligocka's Condition R and admits a ``good" dilation, then the Bergman projection has local $L^p$-Sobolev and H\"older estimates. The good dilation structure is phrased in terms of uniform $L^2$ pseudolocal estimates for the Bergman projection on a family of anisotropic scalings. We conclude the paper by showing that $h$-extendible domains satisfy our hypotheses.
	\end{abstract}
		\thanks{The first author was supported by  the Vietnam National Foundation for Science and Technology Development (NAFOSTED) under grant number 101.02-2023.50}
	\thanks{The second author was supported by a grant from the Simons Foundation (707123, ASR). His work on this project was also supported by the
	National Science Foundation while working as a Program Director in the Division of Mathematical Sciences. Any opinions, findings, conclusions, or 
	recommendations expressed in this material are those of the authors and do not necessarily reflect the views of the National Science Foundation.}
	
	\subjclass{Primary 32A25, Secondary 32T25, 32W05, 32A50}
	
	\maketitle
	\section{Introduction}
	Let $\Om$ be a bounded pseudoconvex domain in $\C^n$ with smooth boundary $\bd\Om$.  The Bergman projection $B = B_\Om$ is one of the  
	fundamental objects associated to $\Om$; it is the orthogonal projection of $L^2(\Om)$ onto the closed subspace of square-integrable holomorphic functions on $\Om$. 
	We can express the Bergman projection via the integral representation
	\[
	Bv(z)=\int_\Om \B(z,w)v(w)\,dw,
	\]
	where $dw$ is the Lebesgue measure on $\Om$, and the integral kernel $\B$ is called the Bergman kernel. 
	
	Since the Bergman projection is defined abstractly on $L^2(\Omega)$, basic questions about $B$ include the local and global regularity and estimates in other spaces, namely
	\begin{enumerate}
		\item $C^\infty$ and $L^2_s$, and
		\item $L^p_s$ ($p\neq 2$) and the spaces of H\"older continuous functions $\Lambda_s$.
	\end{enumerate}
	
	When $\Om$ is of finite type (see \cite{Dang82}), Question 1 has been completely answered \cite{Cat83,Cat87,KoNi65,FoKo72}, and we therefore focus on aspects of Question 2
	that relate directly to the Bergman projection and tools that we can apply in $L^p_s(\Om)$ and H\"older spaces.  Condition R is a well known property introduced
	by Bell and Ligocka \cite{BeLi80} to study the smoothness of biholomorphic mappings and is intimately connected with Question 1. We will introduce a local
	version and refer to the original as global Condition R. Specifically, for a domain $\Om\subset\C^n$, we say that $\Om$ satisfies \emph{global Condition R} if for 
	every $s\ge 0$ there is $M=M_s$ such that 
	\[
	\no{Bu}_{L^2_s(\Om)}\le c_{s,\Om} \no{u}_{L^2_{s+M}(\Om)}
	\]
	for all $u\in L^2_{s+M}(\Om)$. 
	
	Global Condition R suggests the following local version. For a domain $\Om\subset\C^n$ and an open set $U\subset \C^n$, we say that $\Om$ satisfies \emph{ $L^2$ pseudolocal estimates for the Bergman projection in $U$} if 
		for every $s, m\ge 0$ there is $M=M_{s,m}$ and a constant $c = c_{s,M, U, \Om}>0$ such that 
		\begin{equation}\label{eqn:local cond R}
		\no{\chi_1 Bu}^2_{L^2_s(\Om)}\le c\left( \no{\chi_2 u}^2_{L^2_{s+M}(\Om)}+\no{\chi_3 Bu}^2_{L^2_{-m}(\Om)}\right)
		\end{equation}
		for all $u\in L^2_{s+M}(U\cap \Om)\cap L^2(\Om)$, where $\chi_j\in C^\infty_c(U)$, $j=1,2,3$ and $\chi_j\prec \chi_{j+1}$. 
	
	We will use the following notation throughout this paper. For cutoff functions $\chi, \chi'\in C^\infty_c(U)$, 
	we write $\chi\prec \chi'$ if $\chi'=1$ on $\supp(\chi)$. We use the notation $a\les b$ (respectively, $a \ges b$) if
	there exists a global constant $c>0$ so that $a \leq cb$ (respectively, $a \geq cb$).
	Moreover, we will use $\approx$ for the combination of $\les$ and $\ges$. Also, $L^p_s(\Om)$ are the usual $L^p$-Sobolev spaces of order $s$ on $\Om$. The space 
	$L^{p'}_{-s}(\Om)$ is the dual space of $(L^p_s(\Om))_0$, which is the closure of $C^\infty_c(\Om)$ in $L^p_s(\Om)$. Here, $p$ and $p'$ are H\"older conjugates.

	Global Condition R often arises as a consequence of estimates used to prove global regularity for the $\dib$-Neumann operator. 
	In particular, compactness estimates 
	(which themselves are a consequence of Catlin's Property (P) or McNeal's Property ($\tilde P$)) or the existence of a plurisubharmonic defining function both imply the global regularity of
	the $\dbar$-Neumann operator \cite{Cat84,McN02,BoSt91z}. 
	See \cite{Str08, Har11} for more general sufficient conditions for global regularity. 
	
	Similarly, pseudolocal estimates for the Bergman projection are a consequence of the local regularity theory for the $\dib$-Neumann problem. 
	It is classical that both (interior) elliptic and subelliptic estimates for the $\dib$-Neumann problem implies this local property
		\cite{KoNi65,FoKo72}. Ellipticity only holds for interior sets of domains. 
		Subellipticity itself is equivalent to a finite type condition on the boundary \cite{Cat83, Cat87}. Moreover, there are several classes of pseudoconvex domains of infinite type for which this local property holds \cite{Koh02, KhZa12, BaKhZa14, BaPiZa15}.
	
	A positive answer to Question 2 has been obtained when $\Om$ is both of finite type and satisfies one of the following hypotheses:
	\begin{enumerate}[ 1.]
		\item strict pseudoconvexity \cite{FoSt74e, PhSt77}.
		\item pseudoconvexity in $\C^2$ \cite{Chr88, FeKo88a, FeKo88, McN89, NaRoStWa89, ChNaSt92}
		\item pseudoconvexity in $\C^n$ and a Levi-form  with comparable eigenvalues \cite{Koe02}, or one degenerate eigenvalue \cite{Mac88}.
		\item decoupled \cite{FeKoMa90, ChDu06, NaSt06}.  
		\item convexity in $\C^n$ \cite{McN94, McSt94, McSt97}.
	\end{enumerate}
	
	The purpose of this paper is to give a full answer to Question 2 for a class of pseudoconvex domains of finite type 
	that admit a good anisotropic dilation, and these scalings turn out to be closely related to Catlin's multitype. 
	This class of domains includes $h$-extendible domains (defined below) as well as types 1-5 above \cite{Yu94,Yu95}. \\

	Recall that a defining function $\rho$ for a domain $\Omega\subset\C^n$ is a $C^1$ function defined on a neighborhood of $\bar\Omega$ 
	so that $\Om = \{z\in\C^n : \rho(z)<0\}$, $\bd\Om = \{z\in\C^n: \rho(z)=0\}$,
	and $\nabla \rho \neq 0$ on $\bd\Omega$. In this paper, we reserve $r=r_\Om$ for the signed distance to the boundary function.
	\begin{definition} \label{defn:good dilation}
		Let $\Om$ be a pseudoconvex domain in $\C^n$ with smooth boundary $\bd\Om$. Let $p\in \bd\Om$ and $z=(z_1,\dots,z_n)$ be coordinates
		so that $p$ is the origin and $\Re z_1$ is the real normal direction to $\bd\Om$ at $p$. 
		We say that $\Om$ has a \emph{good anisotropic dilation at $p$} if there exist 
		smooth, increasing functions $\phi_j:(0,1]\to \R^+$, $j=1,\dots,n$, so that $\frac{\phi_j(\delta)}{\delta}$ is decreasing, $\phi_1(\delta):=\delta$, and $\phi_j(1)=1$ for $j=1,\dots, n$. Additionally,
		for every small $\delta>0$, the  anisotropic dilation
		$$\hat z=\Phi_\delta(z)=\Big(\frac{z_1}{\phi_1(\delta)},\dots,\frac{z_n}{\phi_n(\delta)}\Big):\C^n\to\C^n$$ satisfies two conditions:
		\begin{enumerate}
			\item For each $j$, the inequality $$\left|\frac{\di r}{\di z_j}(z)\right|\lesssim \frac{\delta}{\phi_j(\delta)}$$
			holds for all  $z\in \Phi_\delta^{-1}(B(0,1))$. 
			\item There exists a neighborhood $U$ of $\hat p$ (independent of $\delta$) such that the Bergman operator $B_\delta$ of the scaled domain $\Om_\delta:=\Phi_\delta(\Om)$  
			satisfies $L^2$ pseudolocal estimates in $U$ with uniform estimates in $\delta$. 
				This means for all $\chi_j\in C^\infty_c(U)$, $j=1,2,3$,
				such that $\chi_1 \prec \chi_2 \prec \chi_3$ and for every $s,m>0$ there exists $M=M_{s,m}$ such that
			\begin{equation}\label{eqn:good Sobolev inequ. re: scaling, delta}
			\no{\chi_1 B_\delta u}_{L^2_{s}(\Om_\delta)}^{2}\le c_{s,m, U}\left(\no{\chi_2  u}_{L^2_{s+M}(\Om_\delta)}^{2} +\no{\chi_3 B_\delta u}^2_{L^2_{-m}(\Om_\delta)}\right)
			\end{equation}
			holds for all $u\in L^2_{s+M}(U\cap \Om_\delta)\cap L^2(\Om_\delta)$, where the constant $c_{s,m, U}$ is independent of $\delta$.
		\end{enumerate}
	\end{definition}

	Our first result contains pointwise estimates for derivatives of the Bergman kernel.  Also, the function $\pi$ maps points in $\Om$ that are near $\bd\Om$ to the closest point of $\bd\Om$.
	\begin{theorem}\label{main1} 
		Let $\Om$ be a pseudoconvex domain in $\C^n$ and  $(p,q)\in (\bar\Om\times \bar\Om)\setminus\{\text{Diagonal of }\bar\Om\times\bar\Om\}$. 
		Assume that either $\pi(p)$ or $\pi(q)$ admits a good anisotropic dilation $\Phi_\delta(z)=(\frac{z_1}{\phi_1(\delta)},\dots,\frac{z_n}{\phi_n(\delta)})$.
		Then 	\[
		\left|\left(\prod_{j=1}^n \frac{\di^{\alpha_j+\beta_j}}{\di p_j^{\alpha_j}\di\bar q_j^{\beta_j}}\right) \B(p,q)\right|\le C_{\alpha,\beta}\prod_{j=1}^n\bigg(\phi_j\Big(|r(p)|+|r(q)|+\sum_{k=1}^n\phi_j^*(|p_j-q_j|)\Big)\bigg)^{-2-\alpha_j-\beta_j}
		\]
		for nonnegative integers $\alpha_j, \beta_j$. The constant $C_{\alpha,\beta}$ is independent of $p,q$  and ${}^*$ denotes the function inversion operator, i.e., $\phi^*(\phi(\delta))=\delta$.
	\end{theorem}

	The second goal of this paper is to establish local $L^p$-Sobolev and H\"older estimates for the Bergman projection.  
	
	\begin{theorem}\label{main2}Let $\Om$ be a smooth, bounded, pseudoconvex domain in $\C^n$ satisfying global Condition $R$. Let  $U$ be an open set
		so that that either $U\subset\subset \Om$ or $\bd\Om\cap U$ is a set of good anisotropic dilation points.
		Then the Bergman projection $B$ is locally regular on the set $U$ in both $L^p_s$ with $s\ge 0$, $p\in (1,\infty)$ and $\Lambda_s$ with $s>0$. 
		
		Namely,  whenever
		$\chi_0, \chi_1 \in C^\infty_c(U)$ with $\chi_0 \prec \chi_1$, there exists constants $c_s, c_{s,p}>0$ so that
		\[
		\no{\chi_0 Bv}_{L^p_s(\Om)}\le c_{s,p}\big(\no{\chi_1 v}_{L^p_s(\Om)}+ \no{v}_{L^p_0(\Om)}\big)
		\]
		for $v\in L^p_s(\Om\cap U)\cap L^p(\Om)$, $s\ge 0$ and $p\in (1,\infty)$; and 
		\[
		\no{\chi_0 Bv}_{\Lambda_s(\Om)}\le c_s \big(\no{\chi_1 v}_{\Lambda_s(\Om)}+ \no{v}_{L^\infty(\Om)}\big)
		\]
		for $v\in \Lambda_s(\Om\cap U)\cap L^\infty(\Om)$ and $s>0$. 
	\end{theorem}
	We remind the reader of the definition of the H\"older spaces $\Lambda_s(\Om)$ below (Definition \ref{defn:Holder}).
	
	Theorem \ref{main2} is only useful if there exist domains which satisfy the hypotheses, and we now show there are large classes of domains which do so.
	Let $\Om$ be a pseudoconvex domain in $\C^{n}$ and $p$ be a boundary point. There are several notions of the ``type" of a point that aim to measure the curvature of $\bd\Omega$ at $p$. Two
	of the most widely known are the
	\begin{itemize}
		\item D'Angelo (multi)-type, $\Delta(p)=(\Delta_{n}(p),\dots, \Delta_1(p))$ where $\Delta_k(p)$ is the $k$-type, which measures the maximal order of contact of $k$-dimensional varieties 
		with $\bd\Om$ at $p$; and
		\item Catlin multitype, $\mathcal M(p)=(m_1(p),\dots,m_n(p))$, where $m_k(p)$  is the optimal weight assigned to the coordinate direction $z_k$.
	\end{itemize} 
	With these definitions, $\Delta_n(p)=m_1(p)=1$. In \cite{Cat87}, Catlin proved
	that $\M(p)\le \Delta(p)$ in the sense that $m_{n-k+1}(p)\le \Delta_k(p)<\infty$ for $1\le k\le n$. The following definition is given by Yu:
	\begin{definition}\label{defn:h-extendible}
		A pseudoconvex domain is called \emph{$h$-extendible at $p$} if $\Delta(p)=\M(p)$. If $\Om$ is $h$-extendible at $p$, $\M(p)$ is called the multitype at $p$. 
		A pseudoconvex domain is called \emph{$h$-extendible} if every boundary point is $h$-extendible.
	\end{definition}
	
	In \cite{Yu94}, Yu proves $h$-extendibility at $p$ is equivalent to the existence of coordinates $z=(z_1,z')$ centered at $p$ and a defining function $\rho$ that can be expanded near $0$ as follows:
	$$\rho(z)=\Re z_1+P(z')+R(z).$$
	Here $P$ is a $(\frac{1}{m_2},\dots,\frac1{m_n})$-homogeneous plurisubharmonic polynomial, i.e., 
	\begin{equation}\label{P}
	P(\delta^{1/m_2}z_2,\dots, \delta^{1/m_n}z_n)=\delta P(z_2,\dots,z_n)
	\end{equation}
	and contains no pluriharmonic terms. The function $R$ is smooth and satisfies 
	\begin{equation}\label{R}
	R(z)=o \left(\sum_{j=1}^n|z_j|^{m_j+\alpha}\right)
	\end{equation}
	for some $\alpha>0$.

	The $h$-extendible property allows for a pseudoconvex domain 
	$\Om$ to be approximated by a pseudoconvex domain from the outside.
	See \cite{BoStYu95,Yu94,Yu95} for a discussion.

	\begin{theorem}\label{cor1}Let $\Om$ be an $h$-extendible, bounded  domain in $\C^n$. Then $\Om$ satisfies global condition
		$R$ and  $\bd\Om$ is a set of good anisotropic dilation points. Consequently, the Bergman projection is locally regular in the spaces
		$L^p_s(\Om)$ with $1< p< \infty$, $s\ge 0$ and $\Lambda_s(\Om)$ with  $s>0$.
	\end{theorem}
		
The proof of Theorem~\ref{cor1} reveals a new property of $h$-extendible points which we record as our final theorem.
\begin{theorem}\label{max-type} Let $\Om$ be a pseudoconvex domain in $\C^n$. Assume that the open set $S\subset b\Om$ is $h$-extendible. Then the function $T$ defined by
\[
T(p) = \sum_{k=1}^n \frac{1}{m_k(p)},\quad \text{for }p\in S,
\]
is lower semicontinuous.
\end{theorem}
\begin{remark}
Since the Catlin multitype takes on a finite number of values on $S$, the lower semicontinuity is equivalent to the following maximality property for $T$:
For every point $p\in S$ there exists a neighborhood $V\subset S$ of $p$ such that for every $q\in V$, $T(p) \leq T(q)$.
\end{remark}

	The paper is organized as follows. In Section \ref{sec:Bergman,Szego smoothness}, we recall results on local $L^2_s$ estimates and $C^\infty$-regularity of the $\dib$-Neumann operator and
	the Bergman projection. 
	In Section \ref{sec:Berg kern est}, we give a proof of Theorem~\ref{main1}. 
	In Section \ref{sec:Proof of Bergman kernel Sobolev estimates}, we prove Theorem~\ref{main2}.
	In Section \ref{sec:Proof of the {Theorem}}, we prove Theorem~\ref{cor1} and Theorem~\ref{max-type}.

	\section*{Acknowledgements}
	The authors wish to express their gratitude to Professor Emil Straube of Texas A\&M University for his comments and suggestions. 

	%
	%
	\section{Uniform estimates on the Bergman kernel}\label{sec:Bergman,Szego smoothness}
	\subsection{The smoothness of kernels: local behavior}
	In this subsection, $\Om$ is a smooth, bounded pseudoconvex domain and $U$ is an open set in which $L^2$ pseudolocal estimates for the Bergman projection hold.
	We start our estimate of the Bergman kernel by proving that $\B(z,w)$ is smooth near the diagonal and satisfies uniform estimates when the points $z$ and $w$ are a uniform distance apart. Throughout the paper we use the notation that if $\alpha = (\alpha_1,\dots,\alpha_n)$ is an $n$-tuple of nonnegative integers, then 
$D^\alpha = \prod_{j=1}^{n}\frac{\di^{\alpha_j}}{\di z_j^{\alpha_j}}$.

	\begin{theorem}\label{bkernel}Let $\Om\subset\C^n$ be a smooth, bounded pseudoconvex domain and $U$ be an open set in $\C^n$. Suppose that $L^2$ pseudolocal estimates for the Bergman projection hold on $U$.
		Then the Bergman kernel is smooth on $((\bar\Om\cap U)\times  (\bar\Om\cap U))\setminus\{\T{Diagonal of $\bd\Om\cap U$}\}$. Moreover, for every $c>0$ and multi-indices $\alpha$ and $\beta$,
		there exists a positive constant $c_{\alpha,\beta,U}$ so that for every $(z,w)\in ((\bar\Om\cap U)\times  (\bar\Om\cap U))$ 
		satisfying 
		\[
		\delta_I(z,w):=|r(z)|+|r(w)|+|z-w|\geq c,
		\]
		then
		\[
		|D^\alpha_pD^\beta_{\bar q} \B(z,w)|\le c_{\alpha,\beta, U}
		\]
		and $c_{\alpha,\beta,U}$ is independent of $z,w$, and $\Om$. 
	\end{theorem}
We refer to $\delta_I(p,q)$ as the \emph{isotropic distance} of $\Om$, though we recognize that $\delta_I(\cdot,\cdot)$ is usually neither isotropic nor a distance function of $\C^n$. We introduce
a ``nonistropic distance" in Lemma \ref{newlm} below.
	
	\begin{proof} We wish to apply $B$ to an approximation of the identity, so we let $\psi \in C^\infty_c(B(0,1))$ where $\psi\geq0$, radial, and $\int_{\C^n} \psi\, dw=1$. Let 
	$w\in\Om\cap U$ and set
		$\psi_t(\zeta) = t^{-2n}\psi((\zeta-w)/t)$. 

		When $z\neq w$, the fact that $\B(z,w)$ is harmonic in $w$ means that for $t$ small enough
		\[
		D^\beta_{\bar w}\B(z,w)= \int_{\Om}  \B(z,\zeta) D^\beta_{\bar w}\psi_t(\zeta)\, d\zeta = (-1)^{|\beta|} \int_{\Om}  \B(z,\zeta) D^\beta_{\bar \zeta}\psi_t(\zeta)\, d\zeta
		=(-1)^{|\beta|} (BD^\beta \psi_t)(z).
		\]
		Since the Bergman operator is locally regular in $C^\infty$, 
		$$D^\alpha_{z}D^\beta_{\bar w}\B(z,w)=(-1)^{|\beta|} (D^\alpha BD^\beta \psi_t)(z)\in C^\infty(\bar\Om\cap U)$$

		The hypothesis $\delta_I(z,w)\geq c$ implies that $|z-w|\geq \frac{c}{3}$,  $|r(z)|\geq \frac{c}{3}$, or $|r(w)|\geq \frac{c}{3}$. \\
		
		{\bf Case 1: $|z-w|\geq \frac{c}{3}$.} We choose $\epsilon$ sufficiently small such that $B(z,2\epsilon)\cap B(w,2\epsilon)=\emptyset$ and $B(z,2\epsilon),B(w,2\epsilon)\subset U$. 
		Let  $\chi_1\prec\chi_2\prec \chi_3$ such that $\chi_1=1$ on $B(z,\epsilon)$ and $\supp(\chi_3)\subset B(z,2\epsilon)$. By the Sobolev Lemma, we have (for $t<\ep/2$)
		\begin{equation}\label{eqn:basic Bergman/Sobolev est}
		|D^\alpha_{z}D^\beta_{\bar w}\B(z,w)|\le \sup_{\xi \in B(z,\epsilon)\cap \bar\Om} |D^\alpha_{z}D^\beta_{\bar w}\B(\xi,w)| \le  c_\alpha\no{\chi_1 BD^\beta \psi_t}_{L^2_{n+1+|\alpha|}(\Om)}.
		\end{equation}
		Using \eqref{eqn:local cond R} with $u=D^\beta \psi_t$, we obtain 
		\begin{align*}
		\no{\chi_1 BD^\beta \psi_t}_{L^2_{n+1+|\alpha|}(\Om)}& \le  c_\alpha\left(\no{\chi_2 D^\beta \psi_t}_{L^2_{s}(\Om)} +\no{\chi_3BD^\beta \psi_t}_{L^2(\Om)}\right)\\
		& =  c_\alpha \no{\chi_3BD^\beta \psi_t}_{L^2(\Om)}
		\end{align*}
		where $c$ depends on $|\alpha|$ and $n$ but not on $\Om$. Here the equality follows from the fact that $\supp(\chi_j)\cap \supp(\psi_t)=\emptyset$. On the other hand, by the density of smooth,
		compactly supported functions in $L^2(\Omega)$,
		\[
		\no{\chi_3 BD^\beta \psi_t}_{L^2(\Om)}=\sup\{|(\chi_3 B D^\beta\psi_t,v)_{L^2(\Om)}|:\no{v}_{L^2(\Om)}\le 1,\ v\in C^\infty_c(\Omega)\}.
		\]
		Using the self-adjointness of $B$ and the pairing of $(L^2_{n+1}(\Omega))_0$ with its dual $L^2_{-(n+1)}(\Omega)$, we have
		\begin{align*}
		|(\chi_3 B D^\beta\psi_t,v)_{L^2(\Om)}|=&|(\psi_t,D^\beta B\chi_3v)_{L^2(\Om)}|\\
		=&|(\psi_t,\tilde\chi_1D^\beta B\chi_3v)_{L^2(\Om)}|\\
		=&\no{\psi_t}_{L^2_{-(n+1)}(\Om)}\no{\tilde\chi_1D^\beta B\chi_3v}_{L_{n+1}^2(\Om)},
		\end{align*}
		where $\tilde\chi_1$ is chosen such that $\tilde\chi_1=1$ on $B(w,\epsilon)$. Additionally, choose $\tilde \chi_j \in C^\infty_0(\C^n)$, $j=2,3$ so that  $\tilde\chi_1\prec \tilde\chi_2\prec \tilde\chi_3$ and $\supp(\tilde\chi_3)\subset B(w,2\epsilon)$.
		Note that this forces $\supp(\tilde\chi_2)\cap \supp(\chi_3)=\emptyset$. Since $\psi_t\to\delta_w$ in $(C^0(\C^n))^*$ and $L^2_{n+1}(\C^n) \subset C^0_0(\C^n)$ by Sobolev's Lemma, it follows
		from duality that 
		\[
		\no{\psi_t}_{L^2_{-(n+1)}(\Om)}< c.
		\]
		for some $c>0$ that is independent of $t$ and $\Om$.
		By a second application of the inequality \eqref{eqn:local cond R}
		for cut-off functions $\tilde\chi_1, \tilde\chi_2$ and $\tilde\chi_3$ 
		and the fact that $\tilde\chi_2 \chi_3=0$ by support considerations, we obtain 
		\begin{eqnarray}\label{2.3}\begin{split}
		\big\|\tilde\chi_1D^\beta B\chi_3v\big\|_{L_{n+1}^2(\Om)}
		&\le c_\beta( \no{\tilde\chi_2 \chi_3v}_{L_{\tilde{s}}^2(\Om)} +\no{\tilde\chi_3 B\chi_3v}_{L^2(\Om)} ) \\
		&=c_\beta \no{\tilde\chi_3 B\chi_3v}_{L^2(\Om)} \leq c_\beta \|B \chi_3 v\|_{L^2(\Om)} \leq c_\beta \|\chi_3 v\|_{L^2(\Om)}\le c_\beta,
		\end{split}
		\end{eqnarray}
		since $B$ is an orthogonal projection on $L^2(\Om)$. Here, $c_\beta$ depends on $\beta$ and $n$ but does not depends on $\Om$.
		
		{\bf Case 2: $|r(z)|\geq \frac{c}{3}$ or $|r(w)|\geq \frac{c}{3}$.} Assume $|r(z)|\geq c$. If $w$ is near the boundary then $|z-w|\geq \frac c2$, and the conclusion follows from Case 1. Otherwise  $z$ is near $w$,
		and we can use the interior elliptic regularity of the $\dib$-Neumann problem and \eqref{eqn:local cond R} (and Sobolev's Lemma, as above) to obtain 
		\[
		|D^\alpha_zD^\beta_{\bar w}\B(z,w)|\le c_{\alpha,\beta}
		\]
		where $c_{\alpha,\beta}$ is independent of both $z,w$ and the diameter of $\Om$ when $|r(z)|, |r(w)|\geq \frac c2$.		
		
	In both cases, we have proven that 
		\[
		|D^\alpha_zD^\beta_{\bar w}\B(z,w)|\le c_{\alpha,\beta}
		\]
		uniformly for  $w\in \Om\cap U$. As a consequence of the $L^2$-Sobolev regularity of the Bergman projection on finite type domains and the 
		Sobolev Embedding Theorem, this inequality still holds for $w\in \bar\Om\cap U$.
		This completes the proof of Theorem~\ref{bkernel}.
	\end{proof}

	\subsection{The smoothness of kernels: local/nonlocal}
	In this subsection we establish smoothness of the Bergman kernel in the case that one point is in 
	a set for which $L^2$  pseudolocal estimates for the Bergman projection hold and the other is arbitrary. 
	We observe that our estimates may depend on diameter of $\Om$, however, we will only apply these estimates in a fixed domain.\\

	\begin{theorem}\label{gbkernel} Let $\Om\subset\C^n$ be a smooth, bounded pseudoconvex domain and $U$ be an open set in $\C^n$. 
	Suppose that $L^2$ pseudolocal estimates for the Bergman projection hold on $U$ and global Condition $R$ holds for $\Om$.
		Then the Bergman kernel is smooth on $((\bar\Om\cap U)\times  \bar\Om)\setminus\{\T{Diagonal of $\bd\Om\cap U$}\}$. Moreover, for fixed $c>0$ and multi-indices $\alpha$ and $\beta$, whenever
		there exists $c_{\alpha,\beta}>0$ so that for every $(z,w)\in ((\bar\Om\cap U)\times \bar\Om)$  satisfying 
		\[
		|z-w|\geq c,
		\] 
		it follows that 
		\[
		|D^\alpha_zD^\beta_{\bar w} \B(z,w)|\le c_{\alpha,\beta}.
		\]	
	\end{theorem}
	
	\begin{proof}
		Adopting the notation and argument from the first part of the proof of Theorem~\ref{bkernel}, we have 
		\begin{align*}
		|D^\alpha_{z}D^\beta_{\bar w}\B(z,w)|&\le\no{\chi_1 BD^\beta \psi_t}_{L^2_{n+1+|\alpha|}(\Om)}\\
		&\leq c_{\alpha, m} \big(\no{\chi_2 D^\beta \psi_t}_{L^2_{s}(\Om)}+\no{\chi_3B D^\beta \psi_t}_{L^2_{-m}(\Om)})\\
		&=c_{\alpha, m}\no{\chi_3 BD^\beta \psi_t}_{L^2_{-m}(\Om)}
		\end{align*}
		where $m\geq 0$ will be chosen later and $c_{\alpha, m}$ depends on $\alpha$ and $m$. 
		However, 
		\begin{align*}
		\no{B D^\beta \psi_t}_{L^2_{-m}(\Om)}&= \sup\{|(B D^\beta\psi_t,v)_{L^2(\Om)}|:\no{v}_{L^2_{m}(\Om)}\le 1\}\\
		&=\sup\{|(\psi_t,D^\beta Bv)_{L^2(\Om)|}:\no{v}_{L^2_{m}(\Om)}\le 1\}\\
		&\le\sup\{\no{\psi_t}_{L^2_{-(n+1)}(\Om)}\no{Bv}_{L^2_{2n+|\beta|}(\Om)}:\no{v}_{L^2_{m}(\Om)}\le 1\}\\
		&\leq c_\beta\sup\{\no{v}_{L^2_{n+1+|\beta|+M}(\Om)}:\no{v}_{L^2_{m}(\Om)}\le 1\}\\
		&\leq c_\beta ,
		\end{align*}
		where the second inequality follows from the facts that $\no{\psi_t}_{L^2_{-2n}(\Om)}\leq C$ {for a constant $c>0$ that is independent of $t$} and the global Condition R with the choice $m\ge n+1+|\beta|+M$. 
	\end{proof}

	\begin{remark} In \cite{Boa87b}, Boas proved a result similar to Theorem~\ref{gbkernel} with the stronger hypothesis that $z\in \bd\Om$ is a point of finite type and Catlin's Property $(P)$ holds. 
	\end{remark}
	
	%
	%
	\section{Proof of Theorem~\ref{main1}}\label{sec:Berg kern est}

	The following lemma follows easily by the definitions. 
	\begin{lemma}\label{vectorscaling}  Let $u$ be a smooth function on $\Om$. For $z\in \Om$, denote $\hat z:=\Phi_\delta(z)$ and $\hat u(\hat z):=u(z)$. Then 
		\[
		\left(\prod_{j=1}^nD_{z_j}^{\alpha_j}\right) u(z)=\left(\prod_{j=1}^n\left(\phi_j(\delta)\right)^{-\alpha_j}\right)\left(\prod_{j=1}^nD_{\hat z_j}^{\alpha_j}\right) \hat u(\hat z).
		\]
	\end{lemma}
	\begin{proof}[Proof of Theorem \ref{main1}] The proof has three steps.
	
		Step 1. We observe that the result is only in question for points close to $\bd\Om$, so we fix $\sigma>0$ and focus on points of distance at most $\sigma$ from $\bd\Om$. 
		Therefore, we fix a point
		$p\in\Om$ with $r(p) >-\sigma$, translate and rotate (unitarily) the domain so that $\pi(p)=0$ and $p$ is on the $\Rre z_1$ axis. Next, we fix a second point $q$ in $B(0, \sigma) \cap \Om$. 
		
		Step 2. We employ a nonisotropic scaling based on the good anisotropic dilation functions $\phi_j$ and a scaling constant $A\geq 1$ that we determine later but will depend only on $\sigma$. We then observe how the
		Bergman kernel behaves under the scaling.
		
		Step 3. We conclude by showing that if $A > \sqrt{n+1}/\sigma$, then $\hat p$ and $\hat q$ are $\hat \Om \cap B(0,\sigma)$. 
		In this case, the (scaled, isotropic) distance between them is bounded away from $0$, independently
		of $p$ and $q$. We can therefore apply Theorem \ref{bkernel} because the constant in Theorem \ref{bkernel} depends only on $\alpha, \beta$, and $B(0,\sigma)$ and NOT on $\hat\Omega$. 
		We now turn to the detailed arguments of Steps 1-3.
		
		Step 1. By Theorem~\ref{bkernel}, we only need to work on the case that $\delta_{I,\Om}(p,q)$ is sufficiently small, say, $\delta_{I,\Om}(p,q) \leq \sigma$ for some fixed $\sigma>0$. 
			Without loss of generality, we can assume that $\pi(p)$ is a point with a good anisotropic dilation 
		$$
		\Phi_\delta(z)=\Big(\frac{z_1}{\phi_1(\delta)},\dots,\frac{z_n}{\phi_n(\delta)}\Big)
		$$ 
		with associated coordinates $z$ and a fixed neighborhood $\hat U = B(0,\sigma)$ of the origin $\pi(p)$ such that $p\in \Re z_1$ and $p,q\in \hat U$.
		Denote $\hat p= \Phi_{\delta}(p)$, $\hat q=\Phi_{\delta}(q)$, and ${\Om}_{\delta}=\Phi_{\delta}(\Om)$.
		Define $\hat r_{\delta}(\hat z):=\frac{1}{\delta}r(\Phi_\delta^{-1}(\hat z))$ for $\hat z\in\C^n$ . Then the function $\hat r_{\delta}$ is a defining function of $\Om_{\delta}$. Moreover, for all $j=1,\dots, n$ we have  
		$$\left|\frac{\di \hat r_{\delta}}{\di \hat z_j}\right|=\left|\frac{\phi_j(\delta)}{\delta}\frac{\di r(\Phi_\delta^{-1}(\hat z))}{\di z_j}\right|\lesssim 1, \quad \T{for all  }\quad \hat z\in \hat U,$$ 
		where the inequality follows by Definition \ref{defn:good dilation}, part 1. 
		In fact, when $j=1$,  the inequality $\lesssim$ can be replaced by the equality $\approx$ since $\Rre z_1$ is the normal direction to $\bd\Om$ at $\pi(p)$
		(see Definition \ref{defn:good dilation}).   
		Thus
		$$\left|\nabla_{\hat z}\hat r_\delta(\hat z)\right|\approx 1, \quad \T{for }\quad \hat z\in \hat U,$$
		uniformly in $\delta$. 	This means that $\hat r_{\delta}(\hat z)$ can be considered as a distance function from  $\Om_{\delta} \cap \hat U$ to $\bd\Om_{\delta}$. \\
		
		Step 2. By the transformation law  for the Bergman kernel under biholomorphic mappings, we have 
		\begin{eqnarray}\label{transformation}
		\B_\Om(p,q)=\det J_\C\{\Phi_\delta(p) \}\B_{\hat \Om}(\hat p,\hat q) \overline{\det J_\C\{\Phi_\delta(q)\}}=\prod_{j=1}^n\left(\phi_j(\delta)\right)^{-2}\B_{\Om_{\delta}}(\hat p,\hat q).
		\end{eqnarray}
		Combining \eqref{transformation} with Lemma~\ref{vectorscaling}, we obtain
		\begin{equation}
		\label{B:formula}
		\left(\prod_{j=1}^n \frac{\di^{\alpha_j+\beta_j}}{\di p_j^{\alpha_j}\di\bar q_j^{\beta_j}}\right) \B(p,q)=\prod_{j=1}^n\left(\phi_j(\delta)\right)^{-2-\alpha_j-\beta_j}\left(\prod_{j=1}^n \frac{\di^{\alpha_j+\beta_j}}{\di \hat p_j^{\alpha_j}\di\overline{\hat q}_j^{\beta_j}}\right)\B_{ \Om_{\delta}}(\hat p,\hat q).
		\end{equation}

		For $A\ge 1$ to be determined later and $\sigma$ suitably small (so the expressions below are defined),	we set 
		$$\delta=\left(A|r(p)|+A|r(q)|+\sum_{j=1}^n\phi_j^*(A|p_j-q_j|)\right)\ge |r(p)|+|r(q)|+\sum_{j=1}^n\phi_j^*(|p_j-q_j|).$$  	
		Since the $\phi_j$'s are increasing,   $$\prod_{j=1}^n\left(\phi_j(\delta)\right)^{-2-\alpha_j-\beta_j}\le \prod_{j=1}^n\left(\phi_j\left(|r(p)|+|r(q)|+\sum_{j=1}^n\phi_j^*(|p_j-q_j|)\right)\right)^{-2-\alpha_j-\beta_j}.$$	
		Thus the proof of this theorem is complete if we show that there exists $C_{\alpha,\beta} >0$ so that 
		\begin{eqnarray}
		\label{new1}	\left(\prod_{j=1}^n \frac{\di^{\alpha_j+\beta_j}}{\di \hat p_j^{\alpha_j}\di\overline{\hat q}_j^{\beta_j}}\right)\B_{ \Om_{\delta}}(\hat p,\hat q) \leq C_{\alpha,\beta}
		\end{eqnarray}
		uniformly in $\hat p$ and $\hat q$. 
		
		Step 3. We are going to apply Theorem~\ref{bkernel} to prove \eqref{new1}. In order to do it, we must to check that 
		$\hat p,\hat q\in \hat U$ and with our choice of $\delta$ that 
		$\delta_{I,\Om_\delta}(\hat p,\hat q) \geq c$ independently of $p$ and $q$ (our choice of $\delta$ will ensure that $\hat p$ and $\hat q$ are sufficiently far apart.)
		We have 
		
		$$|\hat p|^2=|\hat p_1|^2=\left|\frac{\Re p_1}{\delta}\right|^2\le \left|\frac{r(p)}{A|r(p)|}\right|^2= \frac{1}{A^2};$$ 
		and
		$$|\hat p-\hat q|^2=\sum_{j=1}^n|\hat p_j-\hat q_j|^2=\sum_{j=1}^n\left(\frac{|p_j-q_j|}{\phi_j(\delta)}\right)^2\le\sum_{j=1}^n\left(\frac{|p_j-q_j|}{\phi_j(\phi_j^*(A|p_j-q_j|))}\right)^2=\frac{n}{A^2}.$$
		This implies $\hat p,\hat q\in B(0,\frac{\sqrt{n+1}}{A})$. Choosing $A > \frac{\sqrt{n+1}}{\sigma}$, we note that $\hat p,\hat q\in \hat U$.
		Since $\frac{\phi_j(\delta)}{\delta}$ is decreasing and $\delta\ge \phi_j^*(A|p_j-q_j|)$ for $j=1,\dots,n$, it follows
		\[\frac{\phi_j(\delta)}{\delta}\le \frac{\phi_j(\phi_j^*(A|p_j-q_j|))}{\phi_j^*(A|p_j-q_j|)}=\frac{A|p_j-q_j|}{\phi_j^*(A|p_j-q_j|)}.\]
		This is the same as
		\[ \frac{|p_j-q_j|}{\phi_j(\delta)}\ge \frac{\phi_j^*(A|p_j-q_j|)}{A\delta}.\]
		Therefore, the isotropic distance $\delta_{I,\Om_{\delta}}(\hat p,\hat q)$ satisfies
		\begin{eqnarray}
		\begin{split}
		\delta_{I,\Om_{p,\delta}}(\hat p,\hat q)&=| \hat r_{\delta}(\hat p)|+|\hat r_{\delta}(\hat q)|+|\hat p-\hat q|\\
		&= \frac{|r(p)|}{\delta}+\frac{|r(q)|}{\delta}+\sqrt{\sum_{j=1}^{n}\frac{|p_j-q_j|^2}{\phi_j^2(\delta)}}\\
		&\ge\frac{|r(p)|}{\delta}+\frac{|r(q)|}{\delta}+\frac{\sum_{j=1}^n\phi_j^*(A|p_j-q_j|)}{A\sqrt{n}\delta}\\
		 &\ge\frac{|r(p)|}{A\sqrt{n}\delta}+\frac{|r(q)|}{A\sqrt{n}\delta}+\frac{\sum_{j=1}^n\phi_j^*(A|p_j-q_j|)}{A\sqrt{n}\delta}
		&= \frac{1}{A\sqrt{n}}.
		\end{split}
		\end{eqnarray}
		This completes the proof of Theorem~\ref{main1} for the Bergman kernel.
	\end{proof}

	%
	%
	\section{Proof of Theorem~\ref{main2}}\label{sec:Proof of Bergman kernel Sobolev estimates} 
	We first consider the case when $\bar U$ is a compact subset of $\Om$. 
	It is well known that elliptic estimates for the $\dib$-Neumann problem hold for forms with compact support in $U$ and hence  $L^2$ pseudolocal estimates for the Bergman projection hold on $U$. Theorem~\ref{gbkernel} therefore implies that 
	$$D^\alpha_z\left(\chi_0(z) \B(z,w)\right)\le  c_{\alpha, \chi_0, d(\bd\Om, \bd U)}$$
	for every cut off function $\chi_0$ such that $\supp(\chi_0)\subset U$. 
	Thus the operator $D^\alpha \chi_0B$ is continuous in $L^p(\Om)$ for $0<p\le \infty$. Namely, we  get the desired inequality
	$$\no{\chi_0Bv}_{L^p_s(\Om)}\lesssim \no{v}_{L^p_0(\Om)}$$
	for every $s\ge 0$, $p\in (1,\infty]$, and  $v\in L_0^p(\Om)$. For the case that $\bd\Om\cap U=S$ is a set of good anisotropic dilation points, we have the following lemma. 
	\begin{lemma}\label{newlm} Let $V_\eps$ be a compact set of $U$ such that $d(\bd U,\bd V_\eps)\ge \epsilon$. Then there exists $c_\eps>0$ such that 
		$$\left|\Big(\prod_{j=1}^{n}\frac{\di^{\alpha_j}}{\di z_j^{\alpha_j}}\Big) \B(z,w)\right|\le c_{\eps,\alpha} \prod_{j=1}^n\phi_j^{-2-\alpha_j}(\delta_{NI}(z,w))$$
		for $z\in V_\eps\cap \bar\Om$ and $w\in \bar\Om$, where $\alpha=(\alpha_1,\dots,\alpha_n)$ and 
		$$\delta_{NI}(z,w):=|r(z)|+|r(w)|+\sum_{j=1}^n\phi_j^*(|z_j-w_j|).$$
	\end{lemma}
	\begin{proof} Denote $S_\eps=\{z\in \Om: d(z,b\Om)<\epsilon\}$. If $z\in V_\eps\cap S_\eps$,  then $\pi(z)\in \bd\Om\cap U$ is a good anisotropic dilation point
		by hypothesis. By Theorem~\ref{main1}, we have
		\begin{equation}\label{k1}
		\left|\Big(\prod_{j=1}^{n}\frac{\di^{\alpha_j}}{\di z_j^{\alpha_j}}\Big)\right|\le c \prod_{j=1}^n\phi_j^{-2-\alpha_j}(\delta_{NI}(z,w)),\quad \T{for } w\in\bar\Om
		\end{equation}
		Otherwise if $z\in (\Om\cap V_\eps)\setminus{S_\eps}$ then $|r(z)|\ge \eps$. By Theorem~\ref{gbkernel}, we have 
		\begin{equation}\label{k2}
		\left|\Big(\prod_{j=1}^{n}\frac{\di^{\alpha_j}}{\di z_j^{\alpha_j}}\Big) \B(z,w)\right|\le c_{\epsilon,\alpha} \quad \T{for } w\in\bar\Om.
		\end{equation}
		The proof follows from \eqref{k1} and \eqref{k2}.
	\end{proof}
	The remainder of the proof of Theorem~\ref{main2}  uses the ideas of McNeal and Stein \cite{McSt94}, though their hypotheses on the type are global while ours are local.
	We use Lemma \ref{newlm} to overcome this problem.
	
	\subsection{Local $L^p_s$ estimates}\label{subsec:local L^p}
	Let $s \geq 0$ be an integer. Let $\{\zeta_m: m=0,1,\dots,s\}$ be a sequence of cutoff functions in $C^\infty_c(U)$ so that $\zeta_0 = \chi_1$, $\zeta_s = \chi_0$, and
	$\zeta_{m} \prec \zeta_{m-1}$ for all $m=1,\dots,s$.
	For $\epsilon>0$, we define $\psi_\ep \in C^\infty(\C^n\times\C^n)$ so that
	$$\psi_\epsilon(z,w)=\begin{cases}1& \T{if~~} |z-w|<\epsilon,\\
	0 &\T{if~~} |z-w|>2\epsilon.\end{cases}.$$
	We may choose $\epsilon$ sufficiently small such that 
	\begin{eqnarray}\label{supp}\zeta_0(w)=1 \text{ if there exists  $1 \leq m \leq s$ and $z \in \supp\zeta_m$ so that $w \in \supp(\psi_\ep(z,\cdot))$}.
	\end{eqnarray}
	We observe that 
	\begin{eqnarray}\label{4.2}\begin{split}
	\no{\zeta_m Bv}_{L^p_m}^p &\les \sum_{|\alpha|=m} \no{\zeta_mD^\alpha Bv}^p_{L^p_0}+\no{\zeta_{m-1}Bv}^p_{L^p_{m-1}}\\
	&=\sum_{|\alpha|=m} \int_{\Om}\left|\int_{\Om}\zeta_m(z)D_z^\alpha \B(z,w)v(w)\, dw\right|^pdz+\no{\zeta_{m-1}Bv}^p_{L^p_{m-1}}\\
	&\les \sum_{|\alpha|=m}\Bigg[  \int_{\Om}\left|\int_{\Om}\zeta_m(z)D_z^\alpha \B(z,w)\psi_\epsilon(z,w) v(w)dw\right|^p\, dz\\
	&\ \ +\int_{\Om}\left|\int_{\Om\cap \{|z-w|>\epsilon\}}|\zeta_m(z) D_z^\alpha\B(z,w)||v(w)|dw\right|^pdz \Bigg]
	+\no{\zeta_{m-1}Bv}^p_{L^p_{m-1}}\\
	&\les \sum_{|\alpha|=m} \no{B^\alpha_\epsilon v}_{L_0^p}^p+\no{v}^p_{L_0^p}+\no{\zeta_{m-1}Bv}^p_{L^p_{m-1}}
	\end{split}\end{eqnarray}
	where $B^\alpha_\epsilon$ is the operator with integral kernel $\zeta_m(z)(D^\alpha_z\B(z,w))\psi_\epsilon(z,w)$.
	Here the last inequality follows by Theorem~\ref{gbkernel} and consequently  the constant hidden in the final $\les$ depends on $\ep$. To complete the proof of theorem
	for continuity in $L^p$-Sobolev spaces, we need to show  that 
	for every multiindex $\alpha$ with $|\alpha|=m$, 
	\begin{eqnarray}\label{new}
	\no{B^\alpha_\epsilon v}_{L_0^p}\les \no{\zeta_0v}_{L_m^p}.
	\end{eqnarray}
	
	Let $B_0$ be the operator with associated integral kernel 
	$$\B_0(z,w)=\zeta_0(z)\prod_{j=1}^n \phi_j(\delta_{NI}(z,w))^{-2}\zeta_0(w).$$
	
	The proof of \eqref{new} will follow immediately from Lemma \ref{lmBm} and Lemma \ref{lmB0}.
	\begin{lemma}\label{lmBm} Let $\alpha$ be a multiindex of length $m$. Then for $z\in \Omega$,
		\[
		|(B^\alpha_\epsilon) v(z)|\les \sum_{j=0}^m (B_0|(D^j\zeta_0v)|)(z).
		\]
	\end{lemma}
	\begin{lemma}\label{lmB0} The operator
		$$B_0:L^p_0(\Om)\to L^p_0(\Om)$$
		for every $1<p<\infty$.
	\end{lemma}
	
	\begin{proof}[Proof of Lemma \ref{lmBm}]
		Without loss of generality, we translate and rotate (unitarily) $\Om$ so that $U$ is a neighborhood of the origin, and $\Rre \frac{\p}{\p w_1}$ is the (outward) unit normal to $\bd\Om$ at the origin.  Also, denote 
			$w' = (w_2,\dots,w_n)$.
		We can write  
		\begin{align*}
		B^\alpha_\epsilon v(z) &= \int_\Om (\zeta_m(z)D^\alpha_z\B(z,w))\psi_\epsilon(z,w)v(w)\,dw\\
		&= I+II
		\end{align*}
		where 
		\[
		I=(-1)^m\int_\Om\int_0^{3\epsilon}\cdots\int_0^{3\epsilon}
		\frac{d}{d t_m}\cdots \frac{d}{d t_1}\Big(\zeta_m(z)D^\alpha_z\B\big(z,(w_1-(t_1+\cdots + t_m),w')\big)\Big)\psi_\epsilon(z,w)v(w)\,dt_1\cdots dt_m\, dw
		\]
		and 
		\[
		II=\sum_{j=1}^m\int_\Om \zeta_m(z)D^\alpha_z\B(z,(w_1-3\epsilon j,w'))\psi_\epsilon(z,w)v(w)dw. 
		\]
		For $II$, since 
		$|z-(w_1-3\epsilon j,w')|\ge 3\epsilon j-|z-w|\ge \epsilon$ for $j\ge 1$ and $w\in \supp\psi_\epsilon(z,\cdot)$, we can use Theorem \ref{bkernel} to obtain 
		\[
		|II|\les \int_{\Om}|\zeta_m(z)\psi_\epsilon(z,w)v(w)|\, dw\les \int_{\supp(\psi_\ep(z,\cdot))} |\zeta_0(w)v(w)|\, dw\les (B_0|\zeta_0v|)(z),
		\]
		where the second inequality follows by \eqref{supp} and the last one by the bound  $1\les|\B_0(z,w)|$ which follows from the support condition on $\psi_\ep$.
		
		To estimate $I$, we notice that 
		\[
		\frac{d}{d t_m}\cdots \frac{d}{d t_1} \B(z,w_t)= (-1)^m \frac{\di^m}{\di (\Re w_1)^{m}}\B(z,w_t)
		\]
		where $w_t=(w_1-\sum_{j=1}^m t_j,w')$. We can write 
		\[
		\frac{\di}{\di \Re w_1}=T+aL_1,
		\]
		where $a\in C^\infty$ and $T$ is a tangent to $\bd\Om$ acting in $w$. On other hand we know that  $\B(z,w)$ is anti-holomorphic in $w$, so $L_1\B(z,w_t)=0$ (here $L_1$ acts $w$).
		Thus, we have 
		\begin{align*}
		(-1)^m\frac{\di^m}{\di (\Re w_1)^m}\B(z,w_t)=\sum_{j=0}^m a_jT^j \B(z,w_t)
		\end{align*}
		where each $a_j$ is a $C^\infty$-function in $w$. Using integration by parts, we obtain
		\[
		I=\sum_{j=0}^m\int_\Om\int_0^{3\epsilon}\cdots\int_0^{3\epsilon} (D^\alpha_z\B(z,w_t))\left(\zeta_m(z)(T^*)^j(a_j(w)\psi_\epsilon(z,w)v(w))\right)dt_1\cdots dt_m\, dw 
		\]
		where $T^*$ is the $L^2(\Om)$-adjoint of $T$. 
		
		To start the estimate of the integrand on $I$, we use Taylor's theorem and observe
		\[
		r(w_t)=r(w_1-t,w')=r(w)-\frac{\di r(w)}{\di(\Re w_1)}t+\frac{\di^2 r(\tilde w)}{\di^2(\Re w_1)} t^2
		\]
		where $\tilde w$ lies in the segment $[w,w_t]$. Since $\dfrac{\di r(w)}{\di(\Re w_1)}>0$ and $t\in [0,3m\epsilon ]$,  for small $\epsilon$, it follows
		\[
		|r(w_t)|\approx |r(w)|+t.
		\]

		Since $\phi_1(\delta)=\delta$ and $\delta\le \phi(\delta)$ for $j=2,\dots,n$ and any small $\delta\leq 1$, 
		\[
		|D^\alpha_z\B(z,w_t)|\le c_\ep (\delta_{NI}(z,w_t))^{-2-m}\prod_{j=2}^n \phi_j(\delta_{NI}(z,w_t))^{-2}.
		\]
		for $z\in \supp(\zeta_m)$ by Lemma~\ref{newlm}.
		By the definition of $\delta_{NI}(z,w_t)$ and the fact that $(w_t)_j = w_j$ for $j=2,\dots,n$, we have
		\begin{align*}
		\delta_{NI}(z,w_t)&\approx |r(z)|+|r(w_t)|+|z_1-(w_t)_1|+\sum_{j=2}^n\phi_j^*(|z_j-(w_t)_j|)\\
		&\approx |r(z)|+|r(w)|+t+|z_1-(w_t)_1|+\sum_{j=2}^n\phi_j^*(|z_j-w_j|)\\
		&\approx |r(z)|+|r(w)|+t+|z_1-w_1|+\sum_{j=2}^n\phi_j^*(|z_j-w_j|)\\
		&\approx \delta_{NI}(z,w)+t.
		\end{align*}
		Hence, $\phi_j(\delta_{NI}(z,w_t))\ges \phi_j(\delta_{NI}(z,w))$ for $j=2,\dots,n$. 
		
		Next, by Theorem~\ref{main1}, 
		\begin{align}
		\int^{3\epsilon}_0\cdots\int^{3\epsilon}_0|D^\alpha_z\B(z,w_t)|dt_1\cdots dt_m 
		&\les \prod_{j=2}^n \phi_j(\delta_{NI}(z, w))^{-2}\int^{3\epsilon}_0\cdots\int^{3\epsilon}_0\frac{dt_1\cdots dt_m}{(\delta_{NI}(z,w)+\sum_{j=1}^mt_j)^{m+2}}\nn\\
		&\les(\delta_{NI}(z,w))^{-2} \prod_{j=2}^n \phi_j(\delta_{NI}(z, w))^{-2} = \prod_{j=1}^n \phi_j(\delta_{NI}(z, w))^{-2}. \label{eqn:IBP est of varB}
		\end{align}
		
		Moreover, from \eqref{supp} we have  
		\[
		\sum_{j=0}^m|\zeta_m(z)(T^*)^j(a_j(w)\psi_\epsilon(z,w)v(w))|\les\sum_{j=0}^m|D^j_w(\zeta_0(w)v(w))|.
		\]
		Therefore, 
		\[
		|I| \les \int_\Om \sum_{j=0}^m\B_0(z,w)|(D^j\zeta_0v)(w)|\,dw=\sum_{j=0}^m(B_0|D^j(\zeta_0v)|)(z).
		\]
	\end{proof}

	\begin{proof}[Proof of Lemma~\ref{lmB0}]  That $\phi_j''(\delta)<0$ is a consequence of the fact that $\frac{\phi_j(\delta)}\delta$ is decreasing. Therefore,
	$\phi_j(a+b) \geq \frac{1}{2}\left( \phi_j(a) + \phi_j(b)\right)$ which yields
		\begin{align*}
		\phi_j(\delta_{NI}(z,w))&\ges |z_j-w_j|+\phi_j\left(|r(z)|+|r(w)|+\sum_{k=2}^{j-1}\phi_k^*(|z_k-w_k|)\right)
		\end{align*}
		for $j=2,\dots,n$. 
		Thus, for $0\leq\eta<1$ we have 
		\begin{align*}
		I_{\eta}(z)&=\int_{\Om}|\B_0(z,w)||r(w)|^{-\eta}dw\\
		&\les \int_{\Om\cap U}\frac{dw}{|r(w)|^\eta\delta_{NI}^2(z,w)\prod_{j=2}^n\left(\phi_j(\delta_{NI}(z,w))\right)^2}\\
		&\les\int_{0}^{\delta_0}\dots\int_0^{\delta_0}\frac{\rho_2\dots\rho_n dr\,d \rho_2\dots d\rho_n\,dy_1}{r^\eta(y_1+r+|r(z)|
			+\sum_{j=2}^n\phi_j^*(\rho_j))^2\prod_{j=2}^n\left(\rho_j+\phi_j(r+|r(z)|+\sum_{k=2}^{j-1}\phi_k^*(\rho_k))\right)^2}\\
		&\les\int_{0}^{\delta_0}\dots\int_0^{\delta_0}\frac{ dr\,d \rho_2\dots d\rho_n}{r^\eta(r+|r(z)|+\sum_{j=2}^n\phi_j^*(\rho_j))\prod_{j=2}^n\phi_j(r+|r(z)|+\sum_{k=2}^{j-1}\phi_k^*(\rho_k))}
		\end{align*}
		where the second inequality follows by using polar coordinates in $w_j-z_j$ for $j=2,\dots,n$ with $\rho_j:=|w_j-z_j|$ and the variable changes $r:=-r(w)$, $y_1=|\Im z_1-\Im w_1|$. Using the hypotheses that $\phi_j$
		is increasing and $\frac{\phi_j(\delta)}{\delta}$ is decreasing implies that $\frac{\phi_j^*(\delta)}{\delta}$ is increasing
		for $\delta$ sufficiently small. So we may use 
		the argument of Lemma~3.2 in \cite{Kha13} to establish 
		\begin{equation}\label{Kineq}
		\int^{\delta_0}_0\frac{d\rho}{A_j+\phi_j^*(\rho_j)}\les \frac{\phi_j(A_j)}{A_j}
		\end{equation}
		where $A_j=r+|r(z)|+\sum_{k=2}^{j-1}\phi_k^*(\rho_k)$.
		Thus, 
		\begin{align*}
		I_{\eta}(z)
		&\les\int_{0}^{\delta_0}\dots\int_0^{\delta_0}\frac{dr\,d \rho_2\dots d\rho_{n-1}}{r^\eta(r+|r(z)|+\sum_{j=2}^{n-1}\phi_j^*(\rho_j))\prod_{j=2}^{n-1}\phi_j(r+|r(z)|+\sum_{k=2}^{j-1}\phi_k^*(\rho_k))}\\
		&\les\cdots\les \int_0^{\delta_0}\frac{dr}{r^\eta(r+|r(z)|)}\approx \frac{1}{|r(z)|^{\eta}}.
		\end{align*}
		Let $q$ be the conjugate exponent of $p$ and $v\in L^p(\Om)$. An application of H\"older's inequality establishes
		\begin{align*}
		|(B_0 v)(z)|^p=&\left(\int_\Om \B_0(z,w)v(w)dw\right)^p\\
		&\le\left(\int_\Om|\B_0(z,w)||v(w)|^p|r(w)|^{\eta p/q}dw\right)\left(\int_\Om|\B_0(z,w)||r(w)|^{-\eta}dw\right)^{p/q}\\
		&\les\left(\int_\Om|\B_0(z,w)||v(w)|^p|r(w)|^{\eta p/q}dw\right) |r(z)|^{-\eta p/q}.
		\end{align*}
		
		Therefore, 
		\begin{align*}
		\no{B_0v}_p^p&\les\int_\Om\int_\Om|\B_0(z,w)||v(w)|^p|r(w)|^{\eta p/q}|r(z)|^{-\eta p/q}\,dw\,dz.\\
		&\les \int_\Om I_{\eta p/q}(w)|v(w)|^p|r(w)|^{\eta p/q}dw \\
		&\les\int_\Om |v(w)|^pdw=\no{v}^p_p
		\end{align*}
		if $0 < \eta <q/p$. This completes the proof of this lemma.
	\end{proof}
	
	
	\subsection{Local H\"older estimates}
	We consider the classical H\"older spaces. 
	
	\begin{definition} \label{defn:Holder}
		The space $\Lambda_s(\Om)$ is defined by:
		\begin{enumerate}[1.]
			\item For $0<s<1$, 
			\[
			\Lambda_s(\Omega)=\left\{u : \no{u}_{\Lambda_s}:=\no{u}_{L^\infty}+\sup_{z,z+h\in\Omega}\frac{|u(z+h)-u(z)|}{|h|^\alpha}<\infty\right\}.
			\]
			\item For $s>1$ and non-integer,  
			\[
			\Lambda_s(\Om)=\left\{u: \no{u}_{\Lambda_s}:= \no{D^\alpha u}_{\Lambda_{s-[s]}}<\infty, \T{ for all $\alpha$ such that} ~ |\alpha|\le [s]\right\}.
			\]
			Here $[s]$ is the greatest integer less than $s$. 
			\item For $s=1$, 
			\[
			\Lambda_1(\Om)=\left\{u : \no{u}_{\Lambda_1}:=\no{u}_{L^\infty}+\sup_{z,z+h,z-h\in\Omega}\frac{|u(z+h)+u(z-h)-2u(z)|}{|h|}<\infty\right\}.
			\]
			\item For $s>1$ and integer,  
			\[
			\Lambda_s(\Om)=\left\{u: \no{u}_{\Lambda_s}:= \max_{0 \leq |\alpha| \leq [s]}\no{D^\alpha u}_{\Lambda_{1}}<\infty, \T{for all $\alpha$ such that} ~ |\alpha|\le s-1\right\}.
			\]
		\end{enumerate}
	\end{definition}
	
	From \cite[\S3]{McSt94}, we have the following equivalent formulation of the H\"older spaces.
	\begin{proposition}\label{pro4.4} Let $s>0$. A function $u\in \Lambda_s$ if and only if for every $k\in\mathbb N$ with $k>s$,  there are functions $u_k$ so that $u = \sum_{k=1}^\infty u_k$ and 
		\begin{enumerate}[(i)]
			\item $\|u_k\|_{L^\infty(\Omega)} \les 2^{-ks}\no{u}_{\Lambda_s}$
			\item $\|D^m u_k\|_{L^\infty(\Omega)} \les 2^{mk}2^{-ks}\no{u}_{\Lambda_s}$.
		\end{enumerate}
		The existence of $\{u_k\}$ is equivalent to the decomposition $u = g_k+b_k$ where
		\begin{enumerate}
			\item $\no{b_k}_{L^\infty(\Om)}\les 2^{-ks}\no{u}_{\Lambda_s}$
			\item $\no{D^j g_k}_{L^\infty(\Om)}\les 2^{k(j-s)}\no{u}_{\Lambda_s}$, for $j\le m$.
		\end{enumerate}
	\end{proposition}
	
	\begin{proof}In the case that $\Omega = \R^d$ for some $d\in\N$, Stein proves the equivalence of $u\in \Lambda_s$ with properties
		$(i)$ and $(ii)$ holding as a consequence of the pseudodifferential calculus \cite[\S VI.5]{Ste93}. Essentially,
		$u$ is decomposed into $\sum u_k$ using the standard dyadic difference operators. When $\Omega\subset\R^d$, McNeal and Stein point out that the extension theorems in Stein \cite[Chapter VI]{Ste70s} allow us to pass from
		$\Omega$ to $\R^n$.
		
		The equivalence of $(i)$ and $(ii)$ with $(1)$ and $(2)$ is straightforward. Given $u = \sum_{\ell=1}^\infty u_k$, take $b_k = \sum_{\ell=k}^\infty u_k$ and $g_k = \sum_{\ell=1}^{k-1} u_k$. Conversely, given $u = g_k+b_k$, observe that
		$g_k-g_{k+1} = b_{k+1}-b_k$. Consequently, if we take $u_k = g_k-g_{k+1}$, then $u_k$ satisfies the desired estimates.
	\end{proof}
	
	The following proposition is essentially due to Hardy and Littlewood \cite{McSt94}.
	\begin{proposition}\label{HL}Let $s>0$. If $u\in C^\infty(\Om)\cap L^\infty(\Om)$ satisfies
		\[
		\quad |\nabla^m u(z)|\le A|r(z)|^{-(m-s)}~~\T{~~for every ~~} z\in \Om
		\]
		for every $m>s$, then $u\in \Lambda_s(\Om)$ and $\no{u}_{\Lambda_s(\Om)}\les A + \|u\|_{L^\infty(\Omega)}$.
	\end{proposition}

	\begin{proof}[Proof of Theorem~\ref{main2} for local H\"older estimates.]
		Our goal is to establish the estimate
		\begin{eqnarray}\label{holder}
		\no{\chi_0 Bv}_{\Lambda_s(\Om)}\les \no{\chi_1 v}_{\Lambda_s(\Om)}+\no{v}_{L^\infty(\Om)}.
		\end{eqnarray}
		Let $m=[s]+1$.  An application of Proposition \ref{HL} reduces the proof of \eqref{holder} to showing 
		\[
		|\nabla^m \chi_0 Bv(z)|\les |r(z)|^{-(m-s)}\left(\no{\chi_1 v}_{\Lambda_s(\Om)}+\no{v}_{L^\infty(\Om)}\right).
		\]
		We let $\{\zeta_j\}_{j=0}^m$ and $\psi_\ep(z,w)$ be as Section \ref{subsec:local L^p}
		and choose $\epsilon$ sufficiently small such that 
		\[
		\zeta_0=1 \quad \T{on} \quad \bigcup_{z\in \supp (\zeta_j)}\supp(\psi_\epsilon(z,\cdot)),\quad \T{for } 1\le j\le m.
		\]
		Then similarly to \eqref{4.2}, by applying Theorem \ref{gbkernel}
		\begin{align*}
		|\nabla^m\zeta_m Bv(z)|&\les \sum_{|\alpha|=m} |\zeta_m D^\alpha Bv(z)|+|\nabla^{m-1}\zeta_{m-1}Bv(z)| \\
		&\les \sum_{|\alpha|=m} \Big|\int_\Om \zeta_m(z)D^\alpha_z\B(z,w)\psi_\epsilon(z,w)v(w)dw\Big|+\int_\Om |v(w)|dw+|\nabla^{m-1}\zeta_{m-1}Bv(z)| \\
		&\les \sum_{|\alpha|=m} |B_\epsilon^\alpha v(z)|+|\nabla^{m-1}\zeta_{m-1}Bv(z)|+\no{v}_{L^\infty}.
		\end{align*}
		
		To estimate $|B_\epsilon^\alpha v(z)|$, we use the following lemmas.
		\begin{lemma}\label{supnormK}For every $z\in \Om$ and multiindex $\alpha$ of length $m$, we have 
			$$|B_\epsilon^\alpha v(z)|\les |r(z)|^{-m}\no{\zeta_0v}_{L^\infty(\Om)}.$$
		\end{lemma}
		\begin{proof} It follows from the definition of $B_\epsilon^\alpha$, the fact that $\zeta_0 \equiv 1$ on $\supp \zeta_m$, and \eqref{supp}  that
			\[
			|B_\epsilon^\alpha v(z)|\les \no{\zeta_0v}_{L^\infty(\Om)}\int_{\Om }\zeta_0(z)|D^\alpha_z\B(z,w)|\zeta_0(w)\, dw, \quad\T{for $z\in \Om$}.
			\]
			
			Since $z,w\in \supp(\zeta_0)\subset U$, Theorem~\ref{main1} yields
			\begin{align*}
			|D^\alpha_z\B(z,w)|\les& (\delta_{NI}(z,w))^{-m-2}\prod _{j=2}^n\phi_j(\delta_{NI}(z,w))^{-2}\\
			&\les |r(z)|^{-m+\eta}|r(w)|^{-\eta}  (\delta_{NI}(z,w))^{-2}\prod _{j=2}^n\phi_j(\delta_{NI}(z,w))^{-2}
			\end{align*}
			for $z,w\in \Om\cap U$, where $0<\eta<1$.
			Thus, 
			\begin{align*}
			&\int_{\Om}\zeta_0(z)|D^\alpha_z\B(z,w)|\zeta_0(w)dw\les |r(z)|^{-m+\eta}I_\eta(z)\les |r(z)|^{-m}, \quad \T{for $z\in \Om$}.
			\end{align*}
			Here the last inequality follows the estimate of $I_\eta$ in the proof of Lemma~\ref{lmB0}.
		\end{proof}
		
		\begin{lemma}\label{supnormB}
			For every $z\in \Om$ and multiindex $\alpha$ of length $m$, we have 
			$$|B_\epsilon^\alpha v(z)|\les |r(z)|^{-1}\sum_{j=0}^{m-1}\no{D^{j}\zeta_0v}_{L^\infty(\Om)}.$$
		\end{lemma}
		\begin{proof}
			By
				repeating the argument of Lemma~\ref{lmBm} and the estimate leading to (\ref{eqn:IBP est of varB})
				 but integrating by parts only $(m-1)$-times, we are led to the inequality
			$$|B_\epsilon^\alpha  v(z)|\les \sum_{j=0}^{m-1}\no{D^{j}\zeta_0v}_{L^\infty(\Om)}\int_{\Om}\frac{\zeta_0(z)\zeta_0(w)dw}{(\delta_{NI}(z,w))^{3}\prod_{j=2}^n(\phi_j(\delta_{NI}(z,w)))^{2}}.$$
			Also, the estimate of $I_\eta$ with $\eta=0$ in Lemma \ref{lmB0} 
			immediately yields 
			\[
			\int_{\Om }\frac{\zeta_0(z)\zeta_0(w)\,dw}{(\delta_{NI}(z,w))^{3}\prod_{j=2}^n(\phi_j(\delta_{NI}(z,w)))^{2}}\les |r(z)|^{-1}.
			\]
		\end{proof}

		We now return to the proof of Theorem \ref{main2}. Choose $k$ such that $2^{-k}\approx |r(z)|$. Since $\zeta_0v\in \Lambda^s(\Om)$, by Proposition~\ref{pro4.4} there exists $g_k$ and $b_k$ such that
		$$\zeta_0v=g_k+b_k, \quad \T{on}\quad \Om,$$
		where $$\no{b_k}_{L^\infty(\Om)}\les 2^{-ks}\no{\zeta_0 v}_{\Lambda^s(\Om)}=|r(z)|^s\no{\zeta_0 v}_{\Lambda^s(\Om)}$$ and  
		$$\no{D^j g_k}_{L^\infty(\Om)}\les 2^{k(j-s)}\no{\zeta_0 v}_{\Lambda^s(\Om)}=|r(z)|^{-(j-s)}\no{\zeta_0 v}_{\Lambda^s(\Om)}, \quad \T{for~~}j\le m.$$
		Then 
		\begin{align*}
		|B_\epsilon^\alpha v(z)|&\le |B^\alpha_\epsilon \zeta_0^{-1}b_k(z)|+|B^\alpha_\epsilon \zeta_0^{-1}g_k(z)| \\
		&\les |r(z)|^{-m}\no{b_k}_{L^\infty}+|r(z)|^{-1}\sum_{j=0}^{m-1}\no{D^{j}g_k}_{L^\infty}\\
		&\les \no{\zeta_0 v}_{\Lambda_s}\left(|r(z)|^{-m}|r(z)|^{s}+|r(z)|^{-1}\sum^{m-1}_{j=0}|r(z)|^{-(j-s)}\right)\\
		&\les \no{\zeta_0 v}_{\Lambda_s}|r(z)|^{-(m-s)}.
		\end{align*}
		An application of Proposition~\ref{HL} completes the proof.
	\end{proof}
	
	%
	%
	\section{Proof of Theorem \ref{cor1}} \label{sec:Proof of the {Theorem}}
	Our main theorem in this subsection is
	\begin{theorem}\label{h-extendible thm}
The boundary of an bounded $h$-extendible domain is  a set of good anisotropic dilation points.
	\end{theorem}
The proof of this theorem is divided in following four lemmas. In Lemma~\ref{lm1} we prove the condition (1) in Definition \ref{defn:good dilation}.
The proof of the  condition (2) in Definition \ref{defn:good dilation} is divided into Lemma~\ref{lm_n2}, Lemma~\ref{lm_n3} and Lemma~\ref{lm_n4}\\

Throughout this section, $U_o$ is a neighborhood of the origin and $\Om$ is a bounded domain with smooth boundary $b\Om$ in which every boundary point is $h$-extendible.  As discussed in Yu \cite{Yu94, Yu95}, $p\in b\Om$ is $h$-extendible if  there is a  multitype
 $\M(p)=(m_{p,1},m_{p,2},\cdots, m_{p,n})$ with $m_{p,1}=1$, a neighborhood $U_p$ of $p$, a defining function $r_p$ defined in $U_p$, a biholomorphism $H_p: U_p  \to U_o$, 
 (that is, local coordinates associated to $p$) so that $H_p(p)=0$ and $r_{p,1}(z) := r_p(H_p^{-1}(z))$ has the expansion
\begin{equation}\label{eq:n0}
r_{p,1}(z):= \Re z_1+P_p(z')+R_p(z) \qquad\text{for } z=(z_1,z')\in U_o
\end{equation}
where $P_p(z')$ is a $(1/m_{p,2},\dots,1/m_{p,n})$-homogeneous plurisubharmonic polynomial that contains no pluriharmonic terms and $R_q(z)=o(\sigma_p(z))$.
Here,  $$\sigma_p(z): = \sum_{j=1}^n |z_j|^{m_{p,j}}.$$ Thus, there exist 
constants $C>0$ and $\gamma_p>1$ so that the smooth function $R$ satisfies
\[
|R_p(z)| \leq C \sigma_p(z)^{\gamma_p}
\]
(see \cite[Definition 1.4]{Yu94} and the following discussion). Recall that if $f(x) = o(g(x))$ and both functions are smooth, then it follows that
$|\nabla f| = o(|\nabla g|)$.
	
We show that for small $\delta>0$, the map
\begin{equation}\label{new11}
 \Phi_{p,\delta}(z)=\left(\frac{z_1}{\delta},\frac{z_2}{\delta^{1/m_{p,2}}},\cdots, \frac{z_n}{\delta^{1/m_{p,n}}} \right)
\end{equation}
is a good anisotropic dilation at $p$. Note that the homogeneity of $P_p$ means 
\begin{equation}\label{new2}	
P_p\left(\frac{z_2}{\delta^{1/m_{p,2}}},\cdots, \frac{z_n}{\delta^{1/m_{p,n}}}\right)=\delta^{-1} P_p(z_2,\cdots,z_n)
\end{equation}
for $z'=(z_2,\cdots, z_n)\in \C^{n-1}$ and $\delta>0$.

\begin{lemma}\label{lm1}
	The dilation $\Phi_{p,\delta}$ satisfies the condition (1) in Definition   \ref{defn:good dilation}.
	\end{lemma}
\begin{proof}
 Since $|\frac{\di r_{p,1}}{\di z_1}(z)|\approx 1=\frac{\delta}{\delta}$ for $z\in U_o$, we only need to check the first condition in Definition \ref{defn:good dilation} for $j=2,\dots, n$. 
For $\delta>0$ sufficiently small, $\Phi_{p,\delta}^{-1}(B(0,1))\subset U_o$. 
Fix such a $\delta$, and suppose that $z\in \Phi_{p,\delta}^{-1}(B(0,1))$. 
Then there exists $\hat z=(\hat z_1,\hat z') \in B(0,1)$ such that $z=(z_1,z')=\Phi^{-1}_{\delta}(\hat z)$. 
Since $\hat z_j = \frac{z_j}{\delta^{1/m_{p,j}}}$, \eqref{new11} and \eqref{new2} imply that
\[
\frac{\p P_p(\hat z')}{\p\hat z_j} = \frac 1\delta \frac{\p P_p(z')}{\p z_j} \frac{\p z_j}{\p \hat z_j} = \frac{\delta^{1/m_j}}{\delta} \frac{\p P_p(z')}{\p z_j}
\]
from which it follows that $|\frac{\p P_p(z')}{\p z_j}| \les \delta^{1-1/m_{p,j}}$. Since $\gamma>1$, it follows that  

$$\Big|\frac{\di R_p}{\di z_j}(z)\Big|= \Big|\frac{\di }{\di z_j}\left(o(\sigma_p(\Phi^{-1}_{p,\delta}(\hat z)))\right)\Big|\lesssim o(\delta^{1-1/m_{p,j}}).$$
Therefore
$$\Big|\frac{\di r_p}{\di z_j}(z)\Big| \le \Big|\frac{\di P_p }{\di z_j}(z')\Big|+\Big|\frac{\di R_p}{\di z_j}(z)\Big| \lesssim \delta^{1-1/m_{p,j}}$$
for $z\in \Phi_{p,\delta}^{-1}(B(0,1))$. 	We have now established the condition (1) in Definition \ref{defn:good dilation}.
\end{proof}
Denote $E_{p,\delta}=\{z\in \C^n: \sigma_p(z)<\delta\}$  the ellipsoid associated with the multitype $\mathcal M(p)$ with radius $\delta$ and centered at the origin.	Let $q\in H_p^{-1}(E_{p,\delta})\cap b\Om$ and 
\begin{equation}\label{gamma}
\gamma=\min\{\gamma_q: q\in H_p^{-1}(\overline{E}_{p,\delta})\cap b\Om \}.
\end{equation}
Let 
$$\Psi_{q\to p,\delta}=  \Phi_{p,\delta}H_pH_q^{-1} \Phi^{-1}_{q,\delta}.$$
The key point of the second condition in Definition \ref{defn:good dilation} is in the following lemma.

\begin{lemma}\label{lm_n2} For every $t>0$ sufficiently small, there exist positive constants $C$ and $\delta(t)$ such that 
	\begin{equation}\label{det_Jac}
	 |\det J\Psi_{q\to p,\delta}\big|_{(-t,0')}|\le C
	\end{equation}
holds uniformly for  $0<\delta\le \delta(t)$.
\end{lemma}\label{thm_n2}
The proof of this lemma is inspired by the proof of  the main theorem in \cite{Niko02} (See Theorem 6.5 below).
 \begin{proof}[Proof of Lemma \ref{lm_n2}] 
  Recall that $H_p$ is a biholomorphism from $U_p$ to $U_o$, with $H_p(p)=0$, so it can be extended to be a $C^\infty$ diffeomorphism from $\C^n$ to $\C^n$. 
 Define $\Om_{p,\delta}:=\Phi_{p,\delta}H_p(\Om)$ and
$$r_{p,\delta} (z):=\frac{1}{\delta} r_{p,1}(\Phi^{-1}_{p,\delta}(z))=\Re z_1+P_p( z')+O\big(\delta^{\gamma_p-1} \sigma_p(z)^{\gamma_p}\big), 
$$
for $z\in \Phi_{p,\delta}(U_o)$ and $\gamma_p>1$.
Thus, $r_{p,\delta}(z)$ is a defining function for $\Om_{p,\delta}$ in $\Phi_{p,\delta}(U_o)$.
When $\delta\to 0$, 
$$\Om_{p,\delta}\to \Om_{p,0}:=\{z\in \C^n: r_{p,0}(z):=\Re z_1+P_p(z')<0\},$$
where $\Om_{p,0}$ is an associated model for $\Om$ at $p$. It is obvious that $ \Om_{p,1}\cap U_o \equiv H_p(\Om)\cap U_o$. 
For $\alpha, \beta>0$, we define perturbations of  $\Om_{p,0}$  and  $\Om_{p,1}$ by
$$\Om_{p,0}^{\alpha}=\{z\in \C^n:r_{p,0}^{-\alpha}(z):=r_{p,0}(z)-\alpha a_p(z)<0\}$$
where $a_p$ is the bumping function from Yu \cite[Definition 3.3]{Yu95} so that $\Om_{p,0}^\alpha$ is pseudoconvex, and 
$$\Om_{p,1}^{\beta}=\{z\in \C^n:r_{p,1}^{+\beta}(z):=r_{p,1}(z)+\beta<0\}.$$

Let 
$$\Theta_{p}^{\delta, q}:= \Psi_{q\to p,\delta}\Psi_{p\to q,1}=\Phi_{p,\delta}H_pH_q^{-1} \Phi^{-1}_{q,\delta}H_qH_p^{-1}.$$ 
Then $\Theta_{p}^{\delta, q}$ is a biholomorphism from  $U_o$  to its map $ \Theta_{p}^{\delta, q}(U_o)$ since we may choose $U_p$ and $U_q$ such that
$H_q$ is holomorphic on $U_p$ and $H_p$ is holomorphic on $U_q$.

The proof of \eqref{det_Jac} is divided into three steps. \\
 - Step 1. We construct the open set $X$ and $Y$ such that   $\{\Theta_{p}^{\delta, q}\}\in \text{Hol}(X,Y)$ and $Y$ is a taut manifold. \\
 - Step 2. Since $Y$ is a taut manifold, every subsequence of $\{\Theta_{p}^{\delta, q}\}$ either converges normally or diverges compactly. In this step, we prove it is NOT compact divergence. \\
 - Step 3. Using the conclusion in Step 2,  we prove that \eqref{det_Jac} holds.
 
 \medskip
 
 \n\textbf{Proof of Step 1}. In this step we prove that 
for sufficiently small $\alpha, \beta>0$ there exists $\delta_0=\delta(\alpha,\beta)$ such that if $X:= \Om_{p,1}^{\beta}\cap B(0,\beta^{\frac{1}{\gamma}})$ with $\gamma$ as in \eqref{gamma}
and $Y:=\Om_{p,0}^{\alpha}$ then $\Theta_p^{\delta,q}\in \text{Hol}(X,Y)$ for $q\in H^{-1}_p(E_{p,\delta})$ and $0<\delta\le \delta_0$. 

First,  we fix $z_{p,1}\in X$. Then
\begin{equation}\label{k200}
\begin{cases}
|z_{p,1}|\le \beta^{\frac{1}{\gamma}},\\
r_{p,1}(z_{p,1})+\beta<0.
\end{cases}
\end{equation}
Let $z_{q,1}:=H_qH_p^{-1}(z_{p,1})$. 
We have 
\begin{align*}
|z_{q,1}|
\le& \Big|H_qH_p^{-1}(z_{p,1})-H_qH_p^{-1}(0)\Big|+\Big|H_qH_p^{-1}(0)-H_qH_p^{-1}H_p(q)\Big|\quad\text{(since $H_qH_p^{-1}H_p(q)=H_q(q)=0$)}\\
\le& c\left(|z_{p,1}|+\Big|H_p(q)\Big|\right)\\
\le& c\left(|\beta|^{\frac{1}{\gamma}}+\delta^{\frac{1}{m_{p,n}}}\right)
\end{align*} 
where  the last inequality follows by the first inequality of  \eqref{k200} and  the inclusion $H_p(q)\in E_{p,\delta}$. Thus there exist $\delta(\beta)>0$ such that for every $0\le \delta\le \delta(\beta)$, one has
 $|z_{q,1}|\le c\beta^{1/\gamma}$, and hence, 
 $$H_qH_p^{-1}(B(0,\beta^{1/\gamma}))\subset B(0,c\beta^{1/\gamma}).$$
 By our definitions $z_{q,1}=H_qH_p^{-1}(z_{p,1})$ and $r_p(z)\approx r_q(z)$ for $z\in U_o$,  it follows 
 	$$r_{p,1}(z_{p,1})=r_p(H_p^{-1}(z_{p,1}))=r_p(H_q^{-1}(z_{q,1}))\approx r_q(H_q^{-1}(z_{q,1})) = r_{q,1}(z_{q,1})$$ 
Thus 
 $$r_{q,1}(z_{q,1})+c\beta \le c(r_{p,1}(z_{p,1})+\beta).$$
 On the other hand, $\gamma>1$ and $0 \leq \delta < 1$ so
 \begin{align*}
 r_{q,\delta}(z_{q,1})&\leq r_{q,1}(z_{q,1})+|O(\delta^{\gamma-1} \sigma_p(z_{q,1})^\gamma)-O(\sigma_p(z_{q,1})^\gamma)|\\
 &\leq r_{q,1}(z_{q,1})+c\sigma_q^{\gamma}(z_{q,1})\\
  &\leq r_{q,1}(z_{q,1})+c|z_{q,1}|^{\gamma}\\
   &\leq r_{q,1}(z_{q,1})+c\beta.\\
 \end{align*}
 Thus, $$ r_{q,\delta}(z_{q,1})\le c(r_{p,1}(z_{p,1})+\beta).$$ 
 It follows 
  $$H_qH_p^{-1}(\Om_{p,1}^{\beta})\cap B(0,c\beta^{1/\gamma})\subset \Om_{q,\delta}\cap B(0,c\beta^{1/\gamma}).$$
 Therefore, we have 
 $$H_qH_p^{-1}(X)\subset \Om_{q,\delta}\cap B(0,c\beta^{1/\gamma}),$$
and
 $$\Phi_{q,\delta}H_qH_p^{-1}(X)\subset \Om_{q,1}\cap E_{q,c\delta\beta^{1/\gamma}}\subset \Om_{q,1}\cap E_{q,\delta}$$
by requiring $\beta$ to be small enough to satisfy $c\beta^{1/\gamma}\le 1$.
Thus, 
$$H_pH_q^{-1}\Phi_{q,\delta}H_qH_p^{-1}(X)\subset \Om_{p,1}\cap H_pH_q^{-1}(E_{q,\delta})\subset \Om_{p,1}\cap A_\delta$$
where $$A_\delta:= \underset{q\in H_{p}^{-1}(E_{p,\delta})}{\bigcup} H_pH_q^{-1}(E_{q,\delta})$$
It is easy to see that   $A_{\delta}$ tends to the origin as $\delta\to 0$. 
Thus, for every $\alpha>0$, there exists $\delta(\alpha)$ such that 
$$A_\delta\subset \Big\{z\in U_o : \big|O(\sigma_p(z)^\gamma)\big|\le \alpha \sigma_p(z)\Big\},\qquad \text {for $0<\delta \le \delta(\alpha)$}.$$
This implies 
$r_{p,0}^{-\alpha}(z)\le r_{p,1}(z)$ for $z\in A_\delta$
and hence
\begin{equation}\label{eqn:set_inequalities}
H_pH_q^{-1}\Phi_{q,\delta}H_qH_p^{-1}(X)\subset \Om_{p,1}\cap  A_\delta  \subset \Om_{p,0}^{-\alpha}\cap A_\delta\subset \Om_{p,0}^{-\alpha}
\end{equation}
 Since $\Phi_{p,\delta}$ in an automorphism of  $\Om^{-\alpha}_{p,0}$, we obtain 
$$\Theta_{p}^{\delta,q}(X)= \Phi_{p,\delta}H_pH_q^{-1}\Phi_{q,\delta}H_qH_p^{-1}(X) \subset \Phi_{p,\delta}(\Om^{-\alpha}_{p,0})= \Om^{-\alpha}_{p,0}=Y$$ 
for $0<\delta\le \delta(\alpha, \beta)$.\\

 \n\textbf{Proof of Step 2.}
The family $\{\Theta_{p}^{\delta,q}\}_{\delta\in (0,\delta_0], q\in H^{-1}_p(E_{p,\delta})\cap b\Om}\subset H(X,Y)$ is a normal  family since $Y$ is a taut complex manifold by Theorem \ref{thm:h extendible is taut}. 
A consequence of tautness is that every subsequence of $\{\Theta_{p}^{\delta,q}\}_{\delta\in (0,\delta_0], q\in H^{-1}_p(E_{p,\delta})\cap b\Om}$ either converges normally or diverges compactly. For $t\in (0,+\infty)$, let $x_{in}= H_pH_q^{-1}(-t,0')$ and $y_{out} = \Psi_{q\to p,\delta}(-t,0')$. Then 
$$y_{out}= \Psi_{q\to p,\delta}H_qH_p^{-1}(x_{in})=\Theta_{p}^{\delta,q}(x_{in}).$$ 
We will show that compact divergence fails by establishing the existence of $t$ and $c$ that are independent of $\delta$ and such that $x_{in}\subset  X$ and $|y_{out}|\le M$. 
We have 
$$|x_{in}|\le c_1( |t|+\delta^{1/m_{p,n}})$$
and 
$$r_{p,1}(x_{in})\le c_2 r_{q,1}(-t,0)=-c_2 t.$$ 
For $c_1\delta^{1/m_{p,n}}\le \frac{1}{2}$, in order to force  $x_{in}\in X$, we need 
\begin{equation}\label{n202}
\frac{\beta}{c_2}<t< \frac{1}{2c_1}\beta^{1/\gamma}.
\end{equation}
 Since $\gamma>1$, for $\beta<\beta_0=(\frac{c_2}{2c_1})^{1/(\gamma-1)}$, we chose $t$ in the non empty set $(\frac{\beta}{c_2}, \frac{1}{2c_1}\beta^{1/\gamma})$.\\

On the other hand, 
 $$\Psi_{q\to p, \delta}(-t,0')=\Phi_{p,\delta}H_{p}H^{-1}_{q}(-\delta t, 0').$$
 Here the equality follows by $\Phi_{q,\delta}(-t,0')=(-\delta t,0)$. Thus the length $$\Big|H_pH_q^{-1}(-\delta t, 0')-H_p(q)\Big|=\Big|H_pH_q^{-1}(-\delta t, 0')-H_pH_q^{-1}(0)\Big|\le c\delta t.$$
 By the hypothesis $q\in H_p^{-1}(E_{p,\delta})$,  it follows 
 $$H_{p}H^{-1}_{q}(-\delta t, 0'))\in E_{p,\delta(1+ct)}$$
 and hence 
 $$\Phi_{p,\delta}H_{p}H^{-1}_{q}(-\delta t, 0')\in B(0,1+ct)$$
 for some $c$. Thus, 	$\Theta_{p}^{\delta,q}(x_{in})\in \Om_{p,0}^{-\alpha}\cap B(0,M)$ with $M$ independent of $\delta$. This means 
no subsequence of the family $\Theta_{p}^{\delta,q}$ is  compactly divergent. Therefore, it converges uniformly on a compact subsets  of $X$\\
 
 \n\textbf{Proof of Step 3.} 
 Let $\{\delta_j\}_{j=0}^{\infty}\subset (0,\delta_0]$ such that be $\{\delta_j\}_{j=0}^{\infty} \searrow 0$  and $\{q_j\}_{j=0}^\infty$ be a sequence of points in $\C^n$ such that $q_j\in H_p^{-1}(E_{p,\delta})\cap b\Om$.
 Thus, $\{\Theta_{p}^{\delta_j,q_j}(z)\}_{j=0}^\infty$ is a subsequence of the family $\{\Theta_{p}^{\delta,q}\}_{\delta\in (0,\delta_0], q\in H^{-1}_p(E_{p,\delta})\cap b\Om}$. 
 As a  consequence of Step 2, when $S$ is a compact subset of $X=\Om_{p,1}^\beta\cap B(0,\beta^{1/\gamma_p})$, $\{\Theta_{p}^{\delta_j,q_j}(z)\}_{j=0}^\infty$  converges uniformly on $S$. 
  Let $$\Theta_p(z)=\lim_{j\to \infty}\Theta_{p}^{\delta_j,q_j}(z), \quad z\in S.$$ 
 It now follows from the uniform convergence of holomorphic functions on compact sets that $\Theta_p(z)$ is holomorphic on $S$ and
  \begin{equation}\label{n204}
  \det(J\Theta_p)=\lim_{j\to \infty}\det (J\Theta_{p}^{\delta_j,q_j})
    \end{equation}
 uniformly on $S$. 
Recall that $$\Theta_{p}^{\delta_j, q_j}= \Psi_{q_j\to p,\delta_j}\Psi_{p\to q_j,1}$$
This means 
  \begin{equation}\label{n204b}
\det\left(J\Theta_{p}^{\delta_j, q_j}\big|_z\right)=\det\Big(J\Psi_{q_j\to p,\delta_j}\big|_{\tilde z=\Psi_{p\to q_j,1}(z)} \Big)\det\Big( J\Psi_{p\to q_j,1}\big|_z\Big),
\end{equation}
 for $z\in S$.
  We notice that $\Psi_{p\to q_j,1}=H_{q_j}H_{p}^{-1}$ is a transformation of a local coordinates associated to $q_j$ to a local coordinates associated to $p$. Thus, 
    \begin{equation}\label{n204c}
  \lim_{j\to \infty} \Psi_{p\to q_j,1}=\lim_{j\to \infty }H_{q_j}H_{p}^{-1}= G,
  \end{equation}
  where $G$ is holomorphic and its Jacobian has a non-zero determinant on $U_p$ (a set that contains $S$). The reason that $G$ may not be the identity map because $H_{q_j}$ may approach another local coordinate choice associated with the $h$-extendible point $p$ since they are not unique. Combining \eqref{n204}, \eqref{n204b} and \eqref{n204c}, we obtain 
      \begin{equation}\label{n204d}
      \lim_{j\to \infty} \det\Big(J\Psi_{q_j\to p,\delta_j}\big|_{\tilde z=\Psi_{p\to q_j,1}(z)} \Big) =   \det(J\Theta_p\big|_z)\left(\det(JG\big|_z)\right)^{-1},\quad  z\in S.
    \end{equation}
   This implies there exist $N$ and $C$ independent of $j$ such that for all $j\ge N$,  
 $$ |\det(J \Psi_{q_j\to p,\delta_j})\big|_{\tilde z}|\le  C$$
 holds for $\tilde z\in \Psi_{p\to q_j,1}(S)$ and $(\delta,q)\in\{(\delta_j,q_j): j\ge N\}$. A consequence of this argument is
 the existence of $C>0$ and $\delta_0(\beta)>0$ so that if $0 < \delta \leq \delta_0(\beta)$ and $q \in H^{-1}_p(E_{p,\delta})\cap b\Om$ then      
  \begin{equation}\label{n204e}
 |\det(J \Psi_{q\to p,\delta})\big|_{z}|\le  C, \quad \text{ for }z\in \Psi_{p\to q,1}(S)
     \end{equation}
 holds. Moving forward, we assume that $\delta(\beta)$ is small enough that \eqref{n204e} holds.
 
 As in the proof of Step 2, for $0<\beta\le \beta_0$, $0<\delta\le \delta(\beta)$, and $t$ satisfying  \eqref{n202}, it follows 
 $$x_{in}:=\Psi_{q\to p,1}(-t,0')=\Psi^{-1}_{p\to q,1}(-t,0')\in X.$$ 
 So if we choose the compact set $S\subset X$ containing $x_{in}$, we obtain for $0<t<t_0$, there exist $\delta(t)>0$ such that 
    \begin{equation}\label{n204g}
 |\det(J \Psi_{q\to p,\delta})\big|_{(-t,0')}|\le  C, 
 \end{equation}
 hold uniformly in $0<\delta\le \delta(t)$.  This proves Step 3 and also Lemma~\ref{lm_n2}.
 \end{proof}
 
\begin{proof}[Proof of Theorem~\ref{max-type}] Fix $p\in S$ and let $q\in H_p^{-1}(E_{p,\delta})\cap S$. We first notice that if $\mathcal M(p)=(m_{p,1},m_{p,2},\cdots, m_{p,n})$ and $\mathcal M(q)=(m_{q,1},m_{q,2},\cdots, m_{q,n})$ are  multitypes associated to $p$ and $q$, respectively, then 
	$$\det \left(J \Phi_{p,\delta}\big|_z\right)=\delta^{\sum_{k=1}^n \frac{1}{m_{p,k}}}\quad\text{and}\quad \det \left(J \Phi_{q,\delta}\big|_z\right)=\delta^{\sum_{k=1}^n \frac{1}{m_{q,k}}}$$
	for all $z\in \C^n$. Since $\Psi_{q\to p,\delta}=\Phi_{p,\delta}\Psi_{q\to p,1}\Phi_{q,\delta}$ and $|\det\left(J\Psi_{q\to p,1}\right)|$ bounded away from zero,  by \eqref{n204g} we have 
$$	\delta^{\sum_{k=1}^n \frac{1}{m_{p,k}}-\sum_{k=1}^n \frac{1}{m_{q,k}}}\le C'$$
for some constant $C'$ for small $\delta>0$. This implies 
\begin{equation}\label{char}
	\sum_{k=1}^n \frac{1}{m_{p,k}}\le\sum_{k=1}^n \frac{1}{m_{q,k}}.
	\end{equation}
\end{proof}
\begin{remark} The inequality \eqref{char} holds for all $h$-extendible domains. For example, say $\Om$ is the decoupled domain defined by 
	$$\Om=\{z\in \C^n: r(z)=\Rre z_1+\sum_{k=2}^n |z_k|^{2m_k}\}$$
	Then $\mathcal M(0)=(1, 2m_2,\cdots, 2m_{n})$, and
	it is easy to see that for every $q$ in a neighborhood of $0$, the $k$-entry $m_{q,k}$ of $\mathcal M(q)$ is always less than or equal $2m_{k}$. 
	The inequality \eqref{char} holds.
	\end{remark}

Denote ${\mathfrak B}_{\Om_{p,\delta}}$ be the Bergman metric associated to $\Om_{p,\delta}$ and $d_{\Om_{p,\delta}}(z)$ the distance from $z$ to the boundary of $\Om_{p,\delta}$. Let 
$$\kappa_p=\max\big\{\frac{1}{m_{q, n}} : q\in H_p^{-1}(\bar E_{p,\delta})\cap b\Om\big\}.$$
Note that $\kappa$ is bounded uniformly in $\delta$. 
\begin{lemma}\label{lm_n3} The Bergman metric associated to the scaled domain $\Om_{p,\delta}$ has a uniformly lower  bound with the rate $d^{-\kappa}_{\Om_{p,\delta}}(z)$. In particular, one has 
	\begin{equation}\label{bergman_lower_bound}
	{\mathfrak B}_{\Om_{p,\delta}}(z,X) \ge c d_{\Om_{p,\delta}}^{-\kappa_p}(z)|X|
	\end{equation}
	for $z\in U_o$ and $X\in T^{1,0}\C^n\big|_{\Om_{p,\delta}}$, where $c$ is  independent of $\delta$.
\end{lemma}

\begin{proof}[Proof of Lemma \ref{lm_n3}]
Fix $z_{p,\delta}\in \Om_{p,\delta}\cap B(0,1)$ and $X_{p,\delta}=\sum_{j=1}^n X_{p,\delta}^{j} \frac{\di}{\di z_j}\in T^{1,0}\C^n\big|_{\Om_{p,\delta}}$ with $X_{p,\delta}^{j}\in \R$ for $j={1,2, ...,n}$ . Let $q$ be the projection of $H_{p}^{-1}\Phi^{-1}_{p,\delta}(z_{p,\delta})$ to the boundary $b\Om$. Thus, $q\in H_p^{-1} (E_{p,\delta})\cap b\Om$. 
Let $z_{q,\delta}=\Psi_{p\to q,\delta}(z_{p,\delta})$. Then $z_{q,\delta}$ is of the form $z_{q,\delta}= (-t,0')$ where $t=d_{\Om_{q,\delta}}(z_{q,\delta})\approx d_{\Om_{p,\delta}}(z_{p,\delta})$  (as verified in \eqref{approx} below) independently in $\delta$.
By Lemma~\ref{lm_n2},
	\begin{equation}
 |\det J\Psi_{q\to p,\delta} (z_{q,\delta})|\le  C
\end{equation}
for sufficiently small $\delta$. 
Since 
$$J (\Psi_{p\to q,\delta}(z_{p,\delta})) \cdot  J( \Psi_{q\to p,\delta}(z_{q,\delta})) =I_n,$$
we conclude that
\begin{equation}\label{J p to q}
 |\det J (\Psi_{p\to q,\delta}(z_{p,\delta})) |\ge  C.
\end{equation}

By the invariance property of the Bergman metric under biholomorphic mappings, 
$$\mathfrak B_{\Om_{p,\delta}}(z_{p,\delta},X_{p,\delta})=\mathfrak B_{\Om_{q,\delta}}(z_{q,\delta}, X_{q,\delta})=\mathfrak B_{q,1}(z_{q,1}, X_{q,1})$$
where $X_{q,\delta}=(J \Psi_{p\to q,\delta})X_{p,\delta}$ and $X_{q,1}= (J \Phi^{-1}_{q,\delta})X_{q,\delta}$.
 By \cite[Theorem 2]{BoStYu95} (see in Appendix below), it follows
	$$\mathfrak B_{\Om_{q,1}}(z_{q,1},X_{q,1})\ge \frac{1}{2}\Big|\left(J\Phi_{q,\eta}\big|_{\eta =d_{\Om_{p,1}}(z_{q,1})}\right)\cdot X_{q,1}\Big|\mathfrak B_{\Om_{q,0}}(\omega, \hat X)
	\ge c \Big|\left(J\Phi_{q,\eta}\big|_{\eta =d_{\Om_{p,1}}(z_{q,1})}\right) \cdot X_{q,1}\Big|.$$
where $\omega=(-1,0')$, $\hat X$ is a unit vector defined  in Theorem \ref{Thm1_A}, and $c=\inf_{|\hat X|=1} \mathfrak B_{q,0}(\omega, \hat X)>0$. Thus $c$ is independent of $z_{q,1}$ and $X_{q,1}$; 
it depends only on the multitype $\mathcal M(q)$.
We also estimate
	\begin{align*}
\Big|	\left(J\Phi_{q,\eta}\big|_{\eta =d_{\Om_{p,1}}(z_{q,1})}\right)\cdot X_{q,1}\Big|=&\Big|\left(J\Phi_{q,\eta}\big|_{\eta =d_{\Om_{p,1}}(z_{q,1})}\right)\cdot  \big(J \Phi^{-1}_{q,\delta}\big) X_{q,\delta}\Big|\\
=& \Big| \left(J\Phi_{q,\eta}\big|_{\eta = \delta^{-1} d_{\Om_{p,1}}(z_{q,1})}\right)\cdot X_{q,\delta}\Big|\\
=& \left| \sum_{j=1}^n \frac{|X^j_{q,\delta}|^2}{\left(\delta^{-1}d_{\Om_{p,1}(z_{q,1})}\right)^{2/m_{q,j}}}\right|^{\frac{1}{2}}\\
\ge& \frac{|X_{q,\delta}|}{\left(\delta^{-1}d_{\Om_{p,1}(z_{q,1})}\right)^{1/m_{q,n}}}\\
=& \frac{|(J \Psi_{p\to q,\delta})X_{p,\delta}|}{\left(\delta^{-1}d_{\Om_{p,1}(z_{q,1})}\right)^{1/m_{q,n}}}\\
\ge& c\frac{|X_{p,\delta}|}{\left(d_{\Om_{p,\delta}(z_{q,\delta})}\right)^{1/m_{q,n}}}\\
	\end{align*}
where the last inequality follows by \eqref{J p to q} and
\begin{equation}\label{approx}
 d_{\Om_{p,\delta}}(z_p,\delta)\approx \left|r_{p,\delta}(z_{p,\delta})\right| =  \left|\frac{r_{p,1}(z_{p,1})}{\delta}\right| \approx \left| \frac{r_{q,1}(z_{q,1})}{\delta}\right| \approx  \frac{d_{\Om_{q,1}}(z_{q,1})}{\delta}\approx d_{\Om_{q,\delta}}(z_q,\delta).
\end{equation}

Therefore, we conclude that 
	$$B_{\Om_{p,\delta}}(z,X)\ge c d^{-\kappa_p}_{\Om_{p,\delta}} (z)|X|.$$
	\end{proof}
\begin{lemma}\label{lm_n4} The second condition in Definition \ref{defn:good dilation} satisfies. In particular, one has, for $\chi_j\in C_c^\infty (U_o)$ such that $\chi_1\prec \chi_2\prec \chi_3$ and for every $s,m\ge 0$, the estimates
	\begin{equation}\label{eqn:good Sobolev inequ. re: scaling, delta1}
	\no{\chi_1 B_{\Om_{p,\delta}} u}_{L^2_s(\Om_{p,\delta})}^{2}\le c_{s,m}\left(\no{\chi_2  u}_{L^2_{s}(\Om_{p,\delta})}^{2} +\no{\chi_3 B_{\Om_{p,\delta}} u}^2_{L^2_{-m}(\Om_{p,\delta})}\right)
	\end{equation}
	holds for all $u\in L^2_{s}(U_o\cap \Om_{p,\delta})\cap L^2(\Om_{p,\delta})$, where the constant $c_{s,m}$ is independent of $\delta$.
\end{lemma}

\begin{proof}[Proof of Lemma \ref{lm_n4}.] 
	By \cite[Section 5]{KhZa12},	 a lower bound of  the Bergman metric implies the existence of a family of bounded functions $\{\phi^\eta\}_{\eta>0}$ such that 
	$$i\di\dib \phi^\eta(X,X)\ge C \eta^{-2\kappa'}|X|^2\qquad \text{ on $S_\eta\cap V$,}$$
	where $S_\eta=\{\hat z\in \Om_{p,\delta}: -\eta <r_{p,\delta}(\hat z)<0\}$ and any $\kappa'<\kappa$, $C$ is independent of $\delta$ and $\eta$.	
	Thus, by \cite[Theorem 2.1]{Cat87} the subelliptic estimates for $\Om_{p,\delta}$ hold in a neighborhood of the origin with  uniformly in $\delta$. Consequently, the $L^2$ pseudolocal estimates in a neigborhood of the origin hold for the Bergman projection $B_{\Om_{p,\delta}}$ uniformly in $\delta$. 
\end{proof}

\section{Appendix}

\begin{theorem}[Theorem 2 in \cite{BoStYu95}]\label{Thm1_A}
	Let $\Om_{q,1}$ be an $h$-extendible at the boundary point $q$ with multitype $(1, m_{q,2},\cdots, m_{q,n})$ and local model $\Om_{q,0}$. If $\Gamma$ be a nontangential cone in $\Om_{q,1}$ with vertex at $p$, then 
	$$\underset{z\in \Gamma, z\to q}{\lim}\frac{\mathfrak B_{\Om_{q,1}}(z,X)}{\left|\left(J\Phi_{q,\eta}\big|_{\eta =d_{\Om_{q,1}}(z)}\right)\cdot (X)\right|}={\mathfrak B}_{\Om_{q,0}}(\omega, \hat X).$$
	Here $\hat X$ is a unit vector defined by $\hat X=\lim_{z\to p}\dfrac{\left(J\Phi_{q,\eta}\big|_{\eta =d_{\Om_{q,1}}(z)}\right)\cdot (X)}{\left|\left(J\Phi_{q,\eta}\big|_{\eta =d_{\Om_{q,1}}(z)}\right)\cdot (X)\right|}$ and $\omega=(-1,0,\cdots, 0)$.
	\end{theorem}

Let $X$ and $Y$ be two complex manifolds. Denote $\text{Hol}(Y,X)$ the set of holomorphic maps from $Y$ to $X$. Now, we recall the definition of the normal family and taut complex mainfold in \cite{Aba89}
\begin{definition} Let ${\mathcal F}=\{f_\alpha\}_{\alpha\in \mathcal A}$ be a family in $\text{Hol}(X,Y)$. We say that $\mathcal F$ is a normal family if every subsequence $\{f_j\} \subseteq \mathcal F$ either
	\begin{itemize}
		\item ({\it normal convergence}) has a subsequence that converges uniformly
		on compact subsets of $X$; or 
		\item ({\it compact divergence}) has a subsequence $\{f_{j_k}\}$ such that, for each compact $K\subseteq X$  and each compact $L\subseteq Y$, there is a number $N$ so large that $f_{j_k} (K) \cap  L = \emptyset$  whenever $k \ge N$.
	\end{itemize}
\end{definition}
Let $\Delta$ be a unit disk in $\C$. 
\begin{definition}
	A complex manifold $Y$ is {\it taut} if $\text{Hol}(\Delta, Y)$ is a normal family. 
	\end{definition}
\begin{theorem}[Theorem 2.1.2 in \cite{Aba89}] Let $Y$ be a taut complex manifold. Then $\text{Hol}(X,Y)$ is a normal family for every complex manifold $X$. 
\end{theorem}

\begin{theorem}[Theorem 3.1 in \cite{Yu95}]\label{thm:h extendible is taut}
	Every $h$-extendible model is taut.
	\end{theorem}
\begin{theorem}[The main theorem in \cite{Niko02}]\label{thm:Niko02}	Let $\Om_{p,1}$ be an $h$-extendible at the boundary point $p$. Then any two models  for $\Om_{p,1}$ at $p$ are biholomorphically equivalent and determinant of its Jacobian mapping is bounded away from zero in a neighborhood of the origin.  
	\end{theorem}

	\bibliographystyle{alpha}
	\bibliography{mybib}
	
\end{document}